\documentclass[11pt,a4paper,twoside]{amsart}
\usepackage{amssymb,amsmath,amsthm}
\theoremstyle{plain}
\newtheorem{theorem}{Theorem}[section]
\newtheorem{definition}[theorem]{Definition}
\newtheorem{lemma}[theorem]{Lemma}
\newtheorem{corollary}[theorem]{Corollary}
\newtheorem{proposition}[theorem]{Proposition}

\theoremstyle{remark}
\newtheorem{remark}[theorem]{Remark}

\usepackage{hyperref}

\numberwithin{equation}{section}

\newcommand{\C}{\mathbb{C}}
\newcommand{\R}{\mathbb{R}}

\newcommand{\N}{\mathbb{N}}

\newcommand{\F}{\mathcal{F}}

\renewcommand{\Re}{\operatorname{Re}}
\newcommand{\I}{\infty}
\newcommand{\abs}[1]{\left\lvert #1\right\rvert}
\newcommand{\norm}[1]{\left\lVert #1\right\rVert}
\newcommand{\Lebn}[2]{\left\lVert #1 \right\rVert_{L^{#2}}}

\newcommand{\Jbr}[1]{\left\langle #1 \right\rangle}

\newcommand{\IN}{\quad\text{in }}

\def\({\left(}
\def\){\right)}
\def\<{\left\langle}
\def\>{\right\rangle}
\def\le{\leqslant}
\def\ge{\geqslant}

\def\d{{\partial}}
\def\l{\lambda}

\newcommand{\g}{\gamma}

\newcommand{\eps}{\varepsilon}

\newcommand{\ucz}{u_{0,\mathrm{c}}}
\newcommand{\uc}{u_{\mathrm{c}}}

\newcommand{\pst}{p_{\mathrm{St}}}

\begin{document}
\title[Focusing mass-subcritical NLS equation]
{On minimal non-scattering solution for focusing mass-subcritical
nonlinear Schr\"odinger equation
}
\author{Satoshi Masaki}
\date{}
\maketitle
\vskip-5mm
\centerline{Laboratory of Mathematics, Institute of Engineering,
Hiroshima University,}
\centerline{Higashihiroshima Hirhosima, 739-8527, Japan}
\centerline{masaki@amath.hiroshima-u.ac.jp}

\begin{abstract}
We consider time global behavior of solutions to the focusing
mass-subcritical NLS equation in weighted $L^2$ space. We prove that
there exists a threshold solution such that
(i) it does not scatter;
(ii) with respect to a certain scale-invariant quantity,
this solution attains minimum value in all non-scattering solutions.
In the mass-critical case, it is known that ground states are this kind of
threshold solution. However, in our case, it
turns out that the above threshold solution is not a standing-wave solution.
\end{abstract}

\section{Introduction}
In this article, we study the initial value problem
of the following nonlinear Schr\"odinger equation
\begin{equation}\tag{NLS}\label{eq:NLS}
\left\{
\begin{aligned}
	&i \d_t u + \Delta u = \l |u|^{p-1} u, \quad (t,x)\in \R^{1+N}\\
	& u(0)=u_0 \in \F H^1,
\end{aligned}
\right.
\end{equation}
where $N\ge 1$ and $\l=-1$.
The initial data $u_0$ is taken from weighted $L^2$ space
\[
	\F H^1=\F H^1(\R^N) :=\{ f\in L^2(\R^N) \ |\ xf \in L^2 (\R^N) \}
\]
with norm
\[
	\norm{f}_{\F H^1}^2 = \Lebn{f}2^2 + \Lebn{xf}2^2.
\]
We treat short range,
mass-subcritical, and $\F \dot{H}^1$-supercritical case; 
\[
	\max\(1+\frac2N,1+ \frac4{N+2}\) < p < 1 + \frac4N.
\]
It is known that the initial value problem \eqref{eq:NLS}
is globally well-posed in 
$\F H^1$ (See, \cite{NO} and reference therein).
More precisely, for any $u_0 \in \F H^1$ there exists
a solution $u(t)$ such that $e^{-it\Delta}u(t)\in C (\R, \F H^1)$ and 
conserves mass
\[
	M[u(t)] := \Lebn{u(t)}2^2.
\]
Let us refer a solution in this class as an $\F H^1$-solution in what follows.
Remark that, by conservation of mass, $\F H^1$-solutions do not blow up in finite time.
The equation \eqref{eq:NLS} has the following scaling;
if $u(t,x)$ is a solution to \eqref{eq:NLS} then
for all $\omega>0$
\begin{equation}\label{eq:scaling}
	u_\omega (t,x) = \omega^{\frac{2}{p-2}}u(\omega^2 t, \omega x)
\end{equation}
is also a solution to \eqref{eq:NLS}.
Hidano proved existence of small radial solutions of \eqref{eq:NLS}
for data in a scale-invariant homogeneous Sobolev space
$\dot{H}^{-s}$ in \cite{Hi}.

Our concern is global behavior of $\F H^1$-solutions of \eqref{eq:NLS}.
In principle, behavior of a solution
is governed by a balance between
dispersive effect by free Schr\"odinger evolution
and nonlinear interaction.
If the linear dispersive effect becomes dominant for large time, 
then a solution asymptotically behaves as a free solution (scattering).
For the precise definition of scattering, see Definition \ref{def:scatter}, below.
In the case where
 $1+\frac2N< p \le1+\frac4N$ and $p\ge1+\frac4{N+2}$,
it is known that
 an $\F H^1$-solution scatters in $\F H^1$ for both time directions,
provided its initial value is sufficiently small in $\F H^1$ sense
(see \cite{CW1,GOV,NO}).
Notice that smallness of datum gives that of corresponding solution, which
is closely related to weakness of nonlinear effect relative to linear effect.
We also remark that the assumption $p>1+\frac2N$ is essential
for considering scattering phenomenon since it is known that
if $p\le 1+ \frac2N$ then
any non-trivial solution does not scatter (see \cite{Ba,St}).
In the defocusing case $\l=+1$, 
it is further known that, for any data $u_0 \in \F H^1$, the solution $u$ scatters
for both positive and negative time as long as $1+\frac4N>p\ge \pst$, where
\begin{equation}\label{asmp:p}
	\pst:= 1 + \frac{2-N  + \sqrt{N^2+12N+4}}{2N},
\end{equation}
see \cite{Tsu2,CW1,NO}.
The lower bound $p_{\mathrm{St}}$, which is sometimes referred to as a Strauss exponent,
is a root of $Np^2-(N+2)p-2=0$.
Remark that $\pst > \max(1+\frac2N,1+\frac4{N+2})$
for all $N\ge1$.
However, the situation differs
in the focusing case $\l=-1$ because the nonlinear interaction conflicts
with linear dispersive effect.
When datum is not small, that is, when nonlinear effect is not weak,
a solution does not necessarily scatter.
A typical example of non-scattering solution
is a standing wave solution $e^{i\omega^2 t} \varphi_\omega(x)$,
where $\omega>0$ and $\varphi_\omega (x) = \omega^{\frac2{p-1}}\psi(\omega x)$ with
\begin{equation}\label{eq:elliptic}
	-\Delta \psi + \psi = |\psi|^{p-1}\psi.
\end{equation}
There exists a unique positive radial
solution $Q(x)$ to \eqref{eq:elliptic}
called the ground state, provided
$1<p < 1 + \frac{4}{N-2}$ ($1<p<\I$ if $N=1,2$),
see \cite{CazBook} and references therein.
It is characterized as the solution minimizing the energy
\begin{equation}\label{eq:energy}
	E[u(t)]:= \frac12 \Lebn{\nabla u(t)}2^2 -\frac{1}{p+1} \Lebn{u(t)}{p+1}^{p+1}
\end{equation}
among all nontrivial $H^1$-solutions.

In this article, we consider the transition
between scattering and non-scattering.
First main result is existence of
a threshold solution (Theorem \ref{thm:main1}).
More precisely, we show the following:
Let us introduce
\begin{equation}\label{def:ell}
	\ell (f) := \Lebn{f}2^{\frac{N+2}{2}-\frac2{p-1}}
	\Lebn{xf}2^{\frac2{p-1}-\frac{N}2},
\end{equation}
which is well defined for $f\in \F H^1$.
If $p>\pst$ then there exists a special 
$\F H^1$-solution $\uc(t)$  such that
(i) it does not scatter for positive time; 
(ii) $\uc(t)$ attains minimum value of $\ell(u(0))$
among all non-scattering solutions.
From these respects, we refer $\uc(t)$
to as a \emph{minimal non-scattering solution}.
The second assertion gives the following sharp criteria for scattering;
if an $\F H^1$-solution $u(t)$ of \eqref{eq:NLS}
satisfies $\ell(u(0)) < \ell( \uc(0))$ then the solution $u(t)$
scatters for both positive and negative time.
Notice that $\ell(u(0))$ is invariant under the scaling \eqref{eq:scaling}.

Our result is similar to the one in the mass-critical case $p=1+\frac4N$.
Dodson shows in \cite{Do} that if an $L^2$-solution $u(t)$ satisfies
$\Lebn{u(0)}2<\Lebn{Q}2$ then the solution exists globally in time and
scatters in $L^2$ for
both positive and negative time.
In our terminology, this result can be rephrased as ``the ground states are the minimal non-scattering solutions.''
However, in the mass-subcritical case,
it turns out that the minimal non-scattering solution
is not neither a ground state nor
any other standing wave solutions,
which is our second result (Theorem \ref{thm:main2}).

In the energy-critical case $p= 1+\frac4{N-2}$ ($N\ge3$),
global behavior of
solutions of which  initial data belongs to a scale-invariant set 
\begin{align*}
	K= \{ u_0  \in \dot{H}^1 \ | {}&\ 
 E[u_0] < E[W] \}
\end{align*}
is considered,
where $W=(1+\frac{|x|^2}{N(N-2)})^{-\frac{N-2}2}$ is a solution to an elliptic equation $\Delta W + |W|^{\frac4{d-2}}W=0$.
It is shown that $K$ is written as $K=S\cup B$,
where $S$ and $B$ are invariant under the NLS flow
and satisfy $S\cap B=\emptyset$ and $S\ni 0$.
If $u_0 \in S$ then a solution of \eqref{eq:NLS}
with $u|_{t=0}=u_0$ scatters 
for both time directions,
and if $u_0 \in B$ and if either
$u_0\in \F \dot{H}^1$ or $u_0 \in L^2$ is radial
then the solution blows up in finite time.
This is first given by Kenig and Merle \cite{KM} for $3\le N\le 5$
under radial assumption, and is extended by Killip and Visan \cite{KV}
to non-radial $N\ge5$ case.
In the energy-subcritical and mass-supercritical case $1+\frac4N<p< 1+\frac4{N-2}$
($1+\frac4N<p<\I$ if $N=1,2$), 
a similar classification result is obtained for $H^1$-solutions belonging to 
\begin{align*}
	L= \{ u_0  \in H^1 \ | {}&\ M[u_0]^{\frac2{p-1} - \frac{N-2}2 }
 E[u_0]^{ \frac{N}2 - \frac2{p-1} } \\
 &{} <M[Q]^{\frac2{p-1} - \frac{N-2}2 }
 E[Q]^{ \frac{N}2 - \frac2{p-1}} \},
\end{align*}
see \cite{AN,HR} (see also \cite{DHR,FXC} for scattering part).

In this article, we do not work with $H^1$-solutions in the following two respects.
First, one can find an arbitrarily small (in $H^1$ sense)
non-scattering solution in  
the family of ground state solutions.
Hence, it seems difficult to yield any classification 
similar to those in previous results.
Second, a sufficient condition for scattering is  
boundedness of a scale-invariant space-time norm of the form
$\norm{u}_{L^\rho_t(\R,L^\gamma_x(\R^N))}$.
In the mass-subcritical case,
this quantity is not necessarily bounded for $H^1$-solutions.
An answer to these problems is to work with $\F H^1$-solutions.
However, to do so,
we must recast
a concentration-compactness argument, which 
is first carried out in \cite{KM} for $\dot{H}^1$ solutions
and later extended to $H^1$-solutions,
in a form adapted to $\F H^1$-solutions.
This adaption is the main technical issue on this paper.
It then turns out that there is a difference in
a so-called \emph{profile decomposition lemma}. 
More precisely, a bounded sequences in $\F H^1$
are decomposed into a simpler form than those
in the known decompositions of $H^1$- or $\dot H^1$-bounded sequences.

\subsection*{Main results}
We now state our results precisely.
\begin{definition}\label{def:scatter}
Let $X$ be a Banach space.
Let $u$ be a solution to \eqref{eq:NLS} such that $e^{-it\Delta}u(t)\in X$
for any $t\in \R$.
We say $u$ scatters in $X$ for positive time (resp. negative time)
if 
$\lim_{t\to\I} e^{-it\Delta}u(t)$ exists (resp. $\lim_{t\to-\I} e^{-it\Delta}u(t)$ exists) in $X$.
\end{definition}

\begin{definition}
We define  $S_+,S_- \subset \F H^1$ as follows:
\[
	S_+ =\left\{ u_0 \in \F H^1 \ \Biggm|\
	\begin{aligned}
	&\text{solution }u(t)\text{ of }\eqref{eq:NLS}
	\text{ with }u|_{t=0}=u_0\\
	&\text{scatters in }\F H^1\text{ for positive time}
	\end{aligned}
	\right\},
\]
\[
	S_- =\left\{ u_0 \in \F H^1 \ \Biggm|\
	\begin{aligned}
	&\text{solution }u(t)\text{ of }\eqref{eq:NLS}
	\text{ with }u|_{t=0}=u_0\\
	&\text{scatters in }\F H^1\text{ for negative time}
	\end{aligned}
	\right\}.
\]
Further, we define a \emph{ scattering set} $S$ as $S=S_+\cap S_-$.
\end{definition}
The first main result of this article is
\begin{theorem}\label{thm:main1}
Suppose $\l=-1$ and $p\in (\pst, 1+ \frac4N)$, where $\pst$ is given in \eqref{asmp:p}.
Then, there exists $\ucz\in \F H^1$ satisfying the following two properties:
\begin{enumerate}
\item $\ucz \not\in S_+$ \quad (non-scattering);
\item $\displaystyle \ell(\ucz) = \inf_{u_0 \in \F H^1\setminus S} \ell(u_0)$ \quad (minimality).
\end{enumerate}
\end{theorem}
The solution $\uc$ with $\uc|_{t=0}=\ucz$ is the minimal non-scattering solution.
\begin{remark}
The property (2) of Theorem \ref{thm:main1} 
implies that
if $u_0\in \F H^1$ satisfies $\ell (u_0) < \ell(\ucz)$ then $u_0 \in S$.
This criteria is sharp in view of the property (1).
\end{remark}
\begin{remark}
If $u(t)$ is a solution then $\overline{u}(-t)$ is also a solution.
This implies $\overline{u_0}\in S_-$ if and only if $u_0 \in S_+$.
Then, we see that 
 $v_{0,c}:=\overline{\ucz}$ 
satisfies
$v_{0,c} \not\in S_-$
and the property (2) of Theorem \ref{thm:main1}
since $\ell(u(0))=\ell(\overline{u}(0))$.
\end{remark}
\begin{remark}
We do not know whether $\ucz \in S_-$ or not.
\end{remark}

Our proof of Theorem \ref{thm:main1}
is based on a concentration-compactness method introduced in \cite{KM}.
In the mass-critical and -supercritical cases, 
one makes use of a \emph{rigidity theorem} together with concentration-compactness
argument to yield a contradiction from an assumption that
a threshold solution is smaller than the ground state in a suitable sense.
However, in our case, no contradiction can be derived 
from the assumption $\ell(\ucz)<\ell(Q)$
because this assertion is actually true.
\begin{theorem}\label{thm:main2}
For any
standing-wave solution $e^{i\omega^2 t} \varphi_\omega(x)$,
it holds that $\ell(\varphi_\omega)>\inf_{u_0 \in \F H^1\setminus S} \ell(u_0)$.
In particular, 
the minimal non-scattering solution $\uc$ given in Theorem \ref{thm:main1}
is not a standing-wave solution.
\end{theorem}
\begin{remark}
More generally, 
if $\psi \in H^1 \cap \F H^1$ has negative energy then 
$\ell(\psi)>\inf_{u_0 \in \F H^1\setminus S} \ell(u_0)$, in fact (see Theorem \ref{thm:main2'}).
Let us note that,
in the case $1+\frac4N< p < 1+ \frac4{N-2}$ ($1+\frac4N<p<\I$ if $N=1,2$),
the same assumption on $u_0$ yields finite time blowup \cite{Gl}.
\end{remark}

The rest of the article is organized as follows.
In section 2, we briefly recall some basic facts and
well-posedness result in $\F H^1$,
and then give
a necessary and sufficient condition for scattering.
Section 3 is devoted to the proof of 
Theorem \ref{thm:main2}.
Then, we turn to the proof of Theorem \ref{thm:main1}.
In Sections 4 and 5, several tools for
a concentration-compactness method
is established in the framework of weighted $L^2$ space.
We first prove a so-called 
\emph{long time perturbation theory} for $\F H^1$-solutions in Section 4.
Several consequent results such as small data scattering or
oscillating data scattering are also shown there.
The latter says that a solution scatters for both time directions
if its initial data is sufficiently oscillating,
which is specific to the mass-subcritical case and the key for the 
proof of 
the profile decomposition lemma for bounded sequences in $\F H^1$,
considered in Section 5.
Finally, we complete the proof of Theorem \ref{thm:main1} in Section 6.

\section{Necessary and sufficient condition for scattering in $\F H^1$}
Let us begin with introduction of notations.
We write $2^*=\frac{2N}{N-2}$ for $N\ge3$.
Let $p'=\frac{p}{p-1}$ and $\delta(r):=N(\frac12-\frac1r)$.
We say a pair $(q,r)$ is admissible if
$r\in [2,2^*]$ ($r\in[2,\I]$ if $N=1$, $r\in[2,\I)$ if $N=2$)
and $q\delta(r)=2$.
For any admissible pairs $(q,r)$, $(q_1,r_1)$, and $(q_2,r_2)$,
Strichartz' estimates hold.
There exists a constant $C$ such that
\[
	\norm{e^{it\Delta}f}_{L^q(I,L^r)} \le C \Lebn{f}2 
\]
for all interval $I\subset \R$ and $f\in L^2$.
There also exists a constant $C$ such that
\[
	\norm{\int_{t_0}^t e^{i(t-s)\Delta}g(s) ds}_{L^{q_1}(I,L^{r_1})}
	\le C\norm{g}_{L^{q'_2}(I,L^{r'_2})}
\]
for any $t_0 \in I \subset \R$ and $g\in L^{q'_2}(I,L^{r'_2})$.

Throughout this paper, we use the following notation.
\begin{equation}\label{def:rho}
\begin{aligned}
	\rho &{}= \frac{2(p^2-1)}{4-(N-2)(p-1)}, &
	\gamma= p+1, \\
	\widetilde{\rho} &{}= \frac{2(p^2-1)}{Np^2-(N+2)p-2}, &
	\widetilde{\gamma}=p+1.
\end{aligned}
\end{equation}
There are well defined for $p\in (\pst,1+\frac4N)$.
Note that the 
 $L^{\rho}_t(\R,L^\gamma_x(\R^N))$-norm
is invariant under the scaling \eqref{eq:scaling}.
Moreover, the relations
\begin{equation}\label{eq:rhotrho}
	\rho = p \widetilde{\rho}',\quad 
	\gamma = p \widetilde{\gamma}'
\end{equation}
hold.
We also remark that $(\rho,\gamma)$ is not an admissible pair
since 
\begin{equation}\label{eq:nonadm}
	2>\rho \delta(\gamma)
\end{equation}
follows by definition.
We have non-admissible Strichartz' estimate holds for $(\rho,\gamma,
\widetilde{\rho},\widetilde{\gamma})$ (\cite{Ka});
there exists a constant $C>0$ such that,
for any interval $I\subset \R $ and $t_0\in I$, we have
\begin{equation}\label{eq:Str}
	\norm{\int_{t_0}^t e^{i(t-s)\Delta} g(s) ds }_{L^\rho(I,L^\gamma)}
	\le C
	\norm{g}_{L^{\widetilde{\rho}'}(I,L^{\widetilde{\gamma}'})}
\end{equation}
for all $g \in L^{\widetilde{q}'}(I,L^{\widetilde{r}'})$
For more precise conditions of this estimate, see \cite{Fo,Ko,Vi}.
One verifies that if $p>\pst$ then
\begin{equation}\label{eq:acceptable}
	\rho \delta(\gamma)>1,\quad 
	\widetilde{\rho}\delta(\widetilde{\gamma})>1
\end{equation}
holds.

Next, let us collect some facts on $J(t):=x+2it\nabla$.
We have 
$$[i\d_t + \Delta, J(t)]=0$$
and $J(t)$ is written as
\begin{equation}\label{eq:Jt}
	J(t)= e^{i\frac{|x|^2}{2t}} 2it\nabla e^{-i\frac{|x|^2}{2t}} = 
	e^{it \Delta} x e^{-it\Delta}.
\end{equation}
\begin{lemma}\label{lem:decay}
If $r \in [2,2^*]$ ($r \in [2,\I]$ if $N=1$, $r \in[2,\I)$ if $N=2$),
then
\[
	\Lebn{f}\gamma \le C|t|^{-\delta(r)} \Lebn{f}2^{1-\delta(r)} \Lebn{J(t)f}2^{\delta(r)}
\]
for any $t\neq0$. Further, 
there exists a constant $C>0$ such that
\[
	\norm{e^{it\Delta}f}_{L^\rho_t(\R,L^\gamma_x(\R^N))}\le C \ell(f)
\]
holds for $f\in \F H^1$.
\end{lemma}
\begin{proof}
By the Sobolev embedding and the first identity of \eqref{eq:Jt}, we immediately obtain
\begin{align*}
	\Lebn{f}r &{}= \Lebn{e^{-i\frac{|x|^2}{2t}}f}r \le
	C \Lebn{e^{-i\frac{|x|^2}{2t}}f}2^{1-\delta(r)} \Lebn{\nabla e^{-i\frac{|x|^2}{2t}}f}2^{\delta(r)}\\
	&{}\le C |t|^{-\delta(r)} \Lebn{f}2^{1-\delta(r)} \Lebn{J(t)f}2^{\delta(r)}.
\end{align*}
Let us proceed to the proof of second inequality.
We may assume $f\neq0$, otherwise the result is obvious.
Let $T_f=\Lebn{xf}2^2/\Lebn{f}2^2>0$.
Since $e^{it\Delta}$ is unitary on $L^2$,
we see from the identity just shown and the second identity of \eqref{eq:Jt} that
\begin{align*}
	\Lebn{e^{it\Delta}f}\gamma
	&{}\le C |t|^{-\delta(\gamma)} \Lebn{e^{it\Delta}f}2^{1-\delta(\gamma)}
	\Lebn{J(t)e^{it\Delta}f}2^{\delta(\gamma)}\\
	&{}= C |t|^{-\delta(\gamma)} \Lebn{f}2^{1-\delta(\gamma)} \Lebn{xf}2^{\delta(\gamma)}.
\end{align*}
Since $\rho \delta(\gamma)>1$ as noticed in \eqref{eq:acceptable},
\begin{align*}
	\norm{e^{it\Delta}f}_{L^\rho(\R\setminus[-T_f,T_f],L^\gamma)}
	&{}\le C T_f^{\frac1\rho-\delta(\gamma)} \Lebn{f}2^{1-\delta(\gamma)} \Lebn{xf}2^{\delta(\gamma)}\\
	&{}=C \Lebn{f}2^{1-(\frac{2}{\rho}-\delta(\gamma))} \Lebn{xf}2^{\frac{2}{\rho}-\delta(\gamma)}
	=C\ell (f).
\end{align*}
On the other hand, by H\"odler's inequality and Strichartz' estimate, 
\begin{align*}
	\norm{e^{it\Delta}f}_{L^\rho([-T_f,T_f],L^\gamma)}
	&{}\le C T_f^{\frac1\rho-\frac{\delta(\gamma)}2}
	\norm{e^{it\Delta}f}_{L^{\frac2{\delta(\gamma)}}([-T_f,T_f],L^\gamma)}\\
	&{}\le C T_f^{\frac1\rho-\frac{\delta(\gamma)}2} \Lebn{f}2 = C\ell(f),
\end{align*}
which completes the proof.
\end{proof}
We summarize basic well-posedness result in $\F H^1$ and sufficient condition for scattering
in a form suitable for our later use.
This is a consequence of \cite{GOV,NO}.
However, we give a brief proof for readers' convenience.
\begin{proposition}[Global well-posedness, sufficient condition for scattering]\label{prop:GWP}
Suppose $1< p < 1+ 4/N$.
Then, equation \eqref{eq:NLS} is globally well-posed in $L^2$.
Moreover,
the solution scatters in $L^2$ for positive time if
\begin{equation}\label{cond:Lrg}
	\norm{u}_{L^\rho((0,\I),L^{\gamma})} <\I.
\end{equation}
Furthermore, the following hold.
\begin{itemize}
\item If $u_0 \in \F H^1$ (resp. $u_0 \in H^1$)
then the solution satisfies $e^{-it\Delta}u(t) \in C(\R,\F H^1)$
(resp. $u(t) \in C(\R, H^1)$).
\item If $u_0 \in \F H^1$ (resp. $u_0 \in H^1$)
and if \eqref{cond:Lrg} is satisfied 
then $e^{-it\Delta} u(t)$ converges in $\F H^1$
(resp. in $H^1$) as $t\to\I$.
\end{itemize}
\end{proposition}
\begin{remark}
The condition \eqref{cond:Lrg} is  never satisfied with a non-trivial solution,
provided  $p\le 1+ \frac2N$.
When $p<\pst$, $\F H^1$-solutions of the free Scrh\"odinger equation
do not satisfy \eqref{cond:Lrg} in general.
\end{remark}
\begin{proof}
Global well-posedness in $L^2$ is well known (see \cite{Tsu1}).
Namely, for all $u_0\in L^2$, there exists a unique global solution $u(t)$ belonging to
$C(\R,L^2) \cap L^a_{\mathrm{loc}}(\R, L^b(\R^N))$
for any admissible pair $(a,b)$.
Note that $u \in L^\rho_{\mathrm{loc}}(\R, L^\gamma(\R^N))$
follows from $\rho \delta(\gamma)<2$.

Let us prove the scattering.
Assume \eqref{cond:Lrg}.
Take an admissible pair $(q,r)$ so that
\[
	1- \frac2r = \frac{p-1}{\gamma}.
\]
This is possible 
since $\delta(r)=\frac{N(p-1)}{2(p+1)}\in (0,1)$.
A use of Strichartz's estimate gives us
\begin{equation*}
	\norm{u}_{L^{q}(I,L^r)} \le 
	C \Lebn{u(t_0)}2 + C \norm{|u|^{p-1}u}_{L^{q'} (I,L^{r'})}.
\end{equation*}
for any $t_0\in I \subset \R$.
By conservation of mass and H\"older inequality,  
\begin{equation}\label{eq:LWP1}
	\norm{u}_{L^{q}(I,L^r)} \le 
	C \norm{u_0}_{L^2} + C \norm{u}_{L^\rho (I,L^{\gamma})}^{p-1}
	\norm{u}_{L^{q}(I,L^r)}.
\end{equation}
Since $\norm{u}_{L^\rho(\R_+,L^{\gamma})} <\I$ by assumption,
for sufficiently large $\tilde{T}$, we obtain
\[
	\norm{u}_{L^q(([\tilde{T},T] ,L^r )} \le 
	C \norm{u_0}_{L^2} + \frac12
	\norm{u}_{L^q(([\tilde{T},T ],L^r )}
\]
for any $T>\widetilde{T}$, which shows
$\norm{u}_{L^q(([\tilde{T},T] ,L^r )}\le 2C \Lebn{u_0}2 $.
Since $T$ is arbitrary, $\norm{u}_{L^q(([\tilde{T},\I) ,L^r )} <\I$.
Applying Strichartz' estimate again, we see that
$u \in L^a([0,\I),L^b)$ for all admissible pair $(a,b)$.
One then deduces that 
\[
	\psi_+:= u_0 -i\l \int_0^\I e^{-is\Delta}(|u|^{p-1}u)(s) ds 
\]
is well-defined in $L^2$.
Another use of Strichartz's estimate gives us
\[
	\norm{u-e^{it\Delta} \psi_+}_{L^\I(([T,\I ),L^2)} \le C\norm{u}_{L^\rho([T,\I),L^{\gamma})}^{p-1}
 \norm{u}_{L^q(([T,\I ),L^r )} \to 0 
\]
as $T\to\I$, which proves the scattering in $L^2$.

Now we assume that $u_0\in \F H^1$ and prove the first assertion.
Take $t_0 \in \R$ and interval $I\ni t_0$.
Operating $J(t)$ to the integral form of \eqref{eq:NLS}
and applying \eqref{eq:Jt}, we obtain
\[
	J(t)u(t) = e^{i(t-t_0)\Delta} J(t_0)u(t_0)
	-i \int_{t_0}^t e^{i(t-s)\Delta} J(s)(|u|^{p-1}u(s)) ds.
\]
Hence, it holds from Strichartz' estimate that
\begin{equation}\label{eq:LWP2}
	\norm{J(t)u}_{L^{q}(I,L^r)} \le 
	C \Lebn{J(t_0)u(t_0)}2 + C \norm{u}_{L^\rho (I,L^{\gamma})}^{p-1}
	\norm{J(t) u}_{L^{q}(I,L^r)}.
\end{equation}
Since $u \in  L^{2/{\delta(\gamma)}}_{\mathrm{loc}} (\R,L^\gamma) \subset 
  L^\rho_{\mathrm{loc}}(\R,L^\gamma) $,
it follows from \eqref{eq:LWP2}
that the solution satisfies $J(t)u(t) \in L^q(I,L^r)$ at least
on a small interval $I$ around $t=t_0$, provided $J(t_0)u(t_0) \in L^2$.
Further,  for the same interval $I$, we deduce from Strichartz' estimate that
$J(t)u(t) \in C(I,L^2)\cap L^{a}_{\mathrm{loc}}(I,L^{b})$
for any admissible pair $(a,b)$.
A similar argument shows that if the solution 
cannot extend beyond $t=T_1$ as a $\F H^1$-solution, that is,
if $\norm{J(t)u(t)}_{L^q((0,t),L^r)} \to \I$ as $t\uparrow T_1$ holds for some time $T_1>0$
then $\norm{u}_{L^\rho((0,t),L^\gamma)}\to\I$ as $t\uparrow T_1$.
However, such $T_1$ does not exist since $u \in L^\rho_{\mathrm{loc}}(\R,L^\gamma)$
holds from the well-posedness in $L^2$.
Thus, if $u_0 \in \F H^1$ then $J(t)u(t) \in C(\R,L^2) \cap L^{a}_{\mathrm{loc}}(\R,L^{b})$
for any admissible pair $(a,b)$.

Let us proceed to the proof of the second assertion.
Assume $u_0 \in \F H^1$ and \eqref{cond:Lrg}. 
Just as in the proof of $u(t) \in L^{a}((0,\I),L^{b})$,
we see from \eqref{eq:LWP2} that
$J(t)u(t) \in L^{a}((0,\I),L^{b})$ for any admissible pair $(a,b)$.
Then, it holds that
\[
	\norm{J(t)u(t)-e^{it\Delta} x\psi}_{L^\I([T,\I),L^2)}
	\le C_3\norm{u}_{L^\rho([T,\I),L^{\gamma})}^{p-1}
 \norm{J(t)u}_{L^q([T,\I),L^r)} \to 0 
\]
as $T\to\I$. This implies scattering in $\F H^1$ since
\[
	\norm{J(t)u(t)-e^{it\Delta} x\psi}_{L^2}
	=\Lebn{x(e^{-it\Delta}u(t)-\psi)}2.
\]
The case $u_0\in H^1$ is handled in a similar way.
\end{proof}
\begin{proposition}\label{prop:ScatterCond}
Suppose $p\in (\pst,1+\frac4N)$.
Let $u_0\in \F H^1$ and let $u(t)$ be the corresponding global solution
with $u(0)=u_0$. Then, $u$ scatters in $\F H^1$ for positive time
if and only if \eqref{cond:Lrg} holds.
\end{proposition}
\begin{proof}
The ``if part'' is shown in Proposition \ref{prop:GWP}.
We shall prove the ``only if part.''
Assume $u$ scatters in $\F H^1$.
Since $u \in L^{2/\delta(\gamma)}_{\mathrm{loc}}(\R,L^\gamma)\subset L^\rho_{\mathrm{loc}}(\R,L^\gamma)$,
it suffices to show that $u \in  L^\rho((1,\I),L^\gamma)$.
Since $e^{-it\Delta}u(t)$ converges in $\F H^1$ as $t\to\I$, we have
\[
	\norm{e^{-it\Delta}u(t)}_{L^\I((0,\I),\F H^1)} \le C.
\]
Hence, for $t>0$,
\[
	\Lebn{u(t)}{\gamma} \le 
	C t^{-\delta(\gamma)} \norm{e^{-it\Delta}u(t)}_{\F H^1}
	\le Ct^{-\delta(\gamma)}.
\]
Notice that the right hand side belongs to $L^\rho_t((1,\I))$
since $\rho \delta(\gamma)>1$ holds by assumption $p>\pst$,
which completes the proof.
\end{proof}
\begin{remark}
By a similar manner, it is easy to see that $u(t)$ scatters in $\F H^1$ for negative time
if and only if $\norm{u}_{L^\rho((-\I,0),L^\gamma)}<\I$.
\end{remark}
\section{Proof of Theorem \ref{thm:main2}}
Although the existence of the minimal non-scattering solution (Theorem \ref{thm:main1})
has not been proven yet,
we first establish Theorem \ref{thm:main2} by showing that
standing wave solutions do not satisfy the minimality property 
(property (2) of Theorem \ref{thm:main1}).
\begin{lemma}\label{lem:NegativeEnergy1}
Suppose $p\in(\pst,1+\frac4N)$.
Let $u_0 \in H^1 \cap \F H^1$. If $E[u_0]<0$
then $u_0 \in S_+^c\cap S_-^c\subset S^c$.
\end{lemma}
\begin{proof}
It is known that the energy is a conserved quantity.
Hence, $E[u(t)]=E[u_0]<0$ for all $t\in \R$.
It then follows that
$\Lebn{u(t)}{p+1}^{p+1} \ge -(p+1) E[u_0] >0$.
This proves $\norm{u}_{L^\rho((0,\I),L^\gamma)}=\I$ since $\gamma=p+1$ and $\rho<\I$.
Thus, by Proposition \ref{prop:ScatterCond}, $u_0\in S_+^c$.
Similarly, $u_0 \in S_-^c$.
\end{proof}
\begin{theorem}\label{thm:main2'}
Suppose $p\in(\pst,1+\frac4N)$.
If $u_0 \in H^1 \cap \F H^1$ and if $E[u_0]<0$
then $\ell(u_0) > \inf\{ \ell(f) \ |\ f \in \F H^1 \setminus S\}$.
\end{theorem}
\begin{proof}
Suppose $u_0 \in H^1 \cap \F H^1$ and $E[u_0]<0$.
Let $c>0$ be a real number.
Since $E[cu_0]$ is continuous with respect to $c$ and
since $E[u_0]<0$, there exists a number $c_0 \in (0,1)$ such that
$E[c_0u_0]<0$. Then, $c_0u_0 \not\in S$ by Lemma
\ref{lem:NegativeEnergy1}. Moreover,
\[
	\ell ( u_0) > c_0\ell (u_0)=\ell (c_0 u_0) \ge 
	\inf\{ \ell(f) \ |\ f \in \F H^1 \setminus S\},
\]
which is the desired estimate.
\end{proof}
By this theorem, if $u_0$ has negative energy then it cannot be equal to $\ucz$.
The next well known lemma completes the proof of Theorem \ref{thm:main2}.
\begin{lemma}
If $\varphi \in H^1$ is a non-trivial solution to \eqref{eq:elliptic}
then $\varphi \in H^1 \cap \F H^1$ and $E[\varphi]<0$.
\end{lemma}
\begin{proof}
It is known that $\varphi$ decays exponentially. Hence, $\varphi\in \F H^1$.
Further,
\[
	E[\varphi]= \frac{N(p-1)-4}{2N(p-1)}\Lebn{\nabla \varphi}2^2<0 
\]
follows from Pohozaev's identity.
\end{proof}
\section{Long-time perturbation theory and its applications}
In this section, we establish the following proposition.
\begin{proposition}[long-time perturbation theory]\label{prop:lpt}
Suppose $\pst<p<1+\frac4N$.
Let $\widetilde{u}(t,x)$ be a function defined on $[0,\I) \times \R^N$ such that
$e^{-it\Delta}\widetilde{u} \in \F H^1$ for all $t\ge0$.
Define an error function $e$ by
\[
	e := i\d_t \widetilde{u} + \Delta \widetilde{u}
	+ |\widetilde{u}|^{p-1} \widetilde{u}.
\]
Let $u_{0} \in \F H^1 $ and let $u(t)$ be the corresponding unique 
global solution of \eqref{eq:NLS} such that $u(0)=u_{0}$.
\begin{enumerate}
\item For each $A>0$, there exists $\eps_0=\eps_0(A)>0$
such that the following holds;
if
\begin{equation}\label{asmp:ltp1}
\left\{
\begin{aligned}
	&\norm{\widetilde{u}}_{L^{\rho}((0,\I),L^\gamma)} \le A, \\
	&\norm{e}_{L^{\widetilde{\rho}'}((0,\I),L^{\widetilde{\gamma}'})}\le \eps,\\
	&\norm{e^{it\Delta}(u(0)-\widetilde{u}(0)) }_{L^{\rho}((0,\I),L^\gamma)}
	\le \eps
\end{aligned}
\right.
\end{equation}
for $0<\eps\le \eps_0$ then $u$
satisfies $\norm{u}_{L^{\rho}((0,\I),L^\gamma)} \le A+\eps^{\frac1p}$.
\item There exist positive constants $C_0$ and $\delta$ such that if
\begin{align*}
	R:={}& \norm{e^{it\Delta}(u(0)-\widetilde{u}(0)) }_{L^{\rho}((0,\I),L^\gamma)}\\
	&{}+\norm{\widetilde{u}}_{L^{\rho}((0,\I),L^\gamma)}+C_0\norm{e}_{L^{\widetilde{\rho}'}((0,\I),L^{\widetilde{\gamma}'})}
\end{align*}
satisfies $R  < \delta$
then $\norm{u}_{L^{\rho}((0,\I),L^\gamma)}\le 2R$.
\end{enumerate}
\end{proposition}
Let us now recall the following Gronwall-type inequality 
introduced in \cite[Lemma 8.1]{FXC}.
\begin{lemma}\label{lem:gronwall}
Let $1\le \beta < \gamma \le \I$, $0<T\le \I$, and let $f\in L^\rho((0,T))$,
where $1\le \alpha <\I$ is given by the relation $\alpha^{-1}=\beta^{-1}-\gamma^{-1}$.
If $\eta\ge0$ and $\varphi \in L^{\gamma}_{\mathrm{loc}}((0,T))$ satisfy
\[
	\norm{\varphi}_{L^\gamma((0,t))} \le \eta + \norm{f\varphi}_{L^\beta((0,t))}
\]
for all $0\le t\le T$, then it holds that
\[
	\norm{\varphi}_{L^\gamma((0,t))} \le \eta \Phi(\norm{f}_{L^\alpha((0,t))}^\alpha)
\]
for all $0\le t\le T$, where $\Phi(s):=2\Gamma(2^\alpha s+3)$ and $\Gamma$ is the Gamma function.
\end{lemma}
\begin{proof}[Proof of Proposition \ref{prop:lpt}]
Let $w$ be defined by $u-\widetilde{u}=w$.
Then $w$ solves the equation
\[
	i\d_t w+ \Delta w + (|\widetilde{u}+w|^{p-1}(\widetilde{u}+w)
	- |\widetilde{u}|^{p-1}\widetilde{u}) - e = 0.
\]
By non-admissible Strichartz' estimate \eqref{eq:Str} and by \eqref{eq:rhotrho},
we have
\begin{multline}\label{eq:ltp}
	\norm{w}_{L^\rho((0,t), L^\gamma)}
	\le 
	\norm{e^{it\Delta}(u_{0}-\widetilde{u}(0)) }_{L^{\rho}((0,t),L^\gamma)}\\
 	+C\norm{w}_{L^\rho((0,t), L^\gamma)}^{p}
   +C\norm{ \norm{\widetilde{u}(t)}_{L^\gamma}^{p-1}
 	\norm{w(t)}_{L^\gamma} 
	}_{L^{\widetilde{\rho}'}((0,t))}
	+C\norm{e}_{L^{\widetilde{\rho}'}((0,t),L^{\gamma'})}
\end{multline}
for all $t\ge0$.

Let us first prove the first assertion.
Take $A>0$. 
Let $\eps_0(A)>0$ be a number satisfying
\[
	\frac1{\eps_0^{p-1}} \ge 2 \left\{ (2C+1)\Phi\( C^{\frac\rho{p-1}}A^{\rho}\) \right\}^{p},
\]
where $\Phi$ is a function given in Lemma \ref{lem:gronwall}.
Assume \eqref{asmp:ltp1}.
By assumption on $u_{0}$ and $e$,
\begin{multline*}
	\norm{w}_{L^\rho((0,t), L^\gamma)}
	\le (C+1)\eps
	+C\norm{w}_{L^\rho((0,t), L^\gamma)}^p \\+ C
	\norm{ \norm{\widetilde{u}(t)}_{L^\gamma}^{p-1}
 	\norm{w(t)}_{L^\gamma} }_{L^{\widetilde{\rho}'}((0,t))}.
\end{multline*}
Take $T>0$ so that $\norm{w}_{L^\rho((0,T), L^\gamma)}^p \le\eps \le \eps_0$.
Then, we have
\[
	\norm{\varphi}_{L^\rho((0,t))}
	\le (2C+1)\eps
	+ 
	\norm{ f\varphi }_{L^{\widetilde{\rho}'}((0,t))}
\]
for $0\le t \le T$, where $\varphi(t)=\Lebn{w(t)}\gamma$ and
$f(t)=C \norm{\widetilde{u}(t)}_{L^\gamma}^{p-1}$.
Apply Lemma \ref{lem:gronwall} to yield
\begin{align*}
	\norm{\varphi}_{L^\rho((0,t))}
	&{}\le (2C+1)\eps
	\Phi\( 
	\norm{ f }_{L^{\frac\rho{p-1}}((0,t))}^{\frac\rho{p-1}}\) \\
	&{}\le (2C+1)\eps
	\Phi\( C^{\frac\rho{p-1}}
	\norm{ \widetilde{u} }_{L^{\rho}((0,\I))}^{\rho}\)
	\le \eps (2C+1)\Phi\( C^{\frac\rho{p-1}}A^{\rho}\).
\end{align*}
By the assumption on $\eps_0$, 
\[
	\norm{w}_{L^\rho((0,T), L^\gamma)}^p
	\le \eps^{p} \left\{ (2C+1)\Phi\( C^{\frac\rho{p-1}}A^{\rho}\) \right\}^{p}
	\le \frac{\eps}2 \(\frac{\eps}{\eps_0}\)^{p-1} \le \frac{\eps}2.
\]
By this estimate, we conclude that 
\[
	\sup_{T>0} \left\{ \norm{w}_{L^\rho((0,T), L^\gamma)}^p \le \eps  \right\}
	=\I.
\]
Hence, $\norm{w}_{L^\rho((0,\I), L^\gamma)} \le \eps^{1/p}$.
This completes the proof of the first assertion.

We shall proceed to the proof of the second assertion.
Now, let $\delta\in (0,1)$ to be chosen later.
Let $C_0$ be the coefficient of the right hand side of \eqref{eq:ltp}.
Assume $R< \delta$.
Set $R_0:= 
R- \norm{\widetilde{u}}_{L^{\rho}(I,L^\gamma)}$.
Then, \eqref{eq:ltp} gives us
\[
	\norm{w}_{L^\rho((0,t), L^\gamma)}
	\le R_0
	+C\norm{w}_{L^\rho((0,t), L^\gamma)}^p + C \delta^{p-1}
	\norm{w}_{L^{\rho}((0,t),L^\gamma)}
\]
for all $t\ge0$.
Choosing $\delta$ so that $C \delta^{p-1} \le 1/3$, we obtain
\begin{equation}\label{eq:lpt2}
	\norm{w}_{L^\rho((0,t), L^\gamma)}
	\le \frac32 R_0
	+\frac32 C\norm{w}_{L^\rho((0,t), L^\gamma)}^p.
\end{equation}
Now, set $f(x)=\frac32 R_0 + \frac32 C x^p$. Remark that
\[
	f(2R_0)=\frac{3R_0}2 + (3C(2R_0)^{p-1})R_0 \le R_0\(\frac32 + 3C(2\delta)^{p-1}\).
\]
Let $\delta$ be so small that $f(2R_0)< 2R_0$, if necessary.
For such $\delta$, \eqref{eq:lpt2} implies that $\norm{w}_{L^\rho((0,t), L^\gamma)}\le 2R_0$
for every $t\ge 0$, showing $\norm{u}_{L^{\rho}((0,\I),L^\gamma)}
\le \norm{\widetilde{u}}_{L^\rho((0,\I), L^\gamma)} + \norm{w}_{L^\rho((0,\I), L^\gamma)}
< 2R$.
\end{proof}
We introduce two consequent results.
The first one is small data scattering.
\begin{corollary}[small data scattering]\label{cor:sds}
Suppose $\pst<p<1+\frac4N$.
Let $u_0\in \F H^1$.
There exists $\eta_0>0$ such that
if $\norm{e^{it\Delta}u_0}_{L^\rho((0,\I),L^\gamma)} \le \eta_0$
then
$u_0 \in S_+$ and
the corresponding solution $u$  
 satisfies
\begin{equation}\label{eq:sds}
\norm{u}_{L^\rho((0,\I),L^\gamma)} \le 2 \norm{e^{it\Delta}u_0}_{L^\rho((0,\I),L^\gamma)}.
\end{equation}
Further,
there also exists $\eta_1>0$ such that
if $\ell(u_0) \le \eta_1$ then 
the same conclusion holds.
\end{corollary}
\begin{proof}
We apply the above proposition with $t_0=0$ and $\widetilde{u}\equiv0$.
Then, the former part follows immediately.
The latter part is a consequence of Lemma \ref{lem:decay}.
\end{proof}
\begin{remark}\label{rmk:sds}
A similar result holds for negative time and for both time direction.
We omit details.
\end{remark}
We next show that the solution of \eqref{eq:NLS} scatters for both
time directions if the initial is sufficiently ``oscillating.''
This result, which is an extension of \cite{CW1}, plays an important
role in the proof of the profile decomposition lemma.
\begin{proposition}[oscillating data scattering]\label{prop:ODS}
Suppose $\pst<p<1+\frac4N$.
For any $\psi\in \F H^1$ and any $\eps>0$, 
there exists $b_0>0$ such that
if $|b|>b_0$ then
\[
	\norm{e^{it\Delta} e^{ib|x|^2} \psi}_{L^\rho (\R,L^\gamma)} \le \eps.
\]
In particular, 
for any $\psi \in \F H^1$ there exists $b_1$ such that
if $|b|>b_1$ then $e^{ib|x|^2} \psi \in S$.
\end{proposition}
\begin{proof}
The latter half is an immediate consequence of the former,
by means of Corollary \ref{cor:sds} (see also Remark \ref{rmk:sds}).
It is known that
\[
	\norm{e^{it\Delta} e^{ib|x|^2} \psi}_{L^\rho ((0,\I),L^\gamma)}
	\to0
\]
as $b\to\I$
(see \cite{CW1} and \cite[Theorem 6.3.4]{CazBook}).
By symmetry, this also implies
$\norm{e^{it\Delta} e^{ib|x|^2} \psi}_{L^\rho ((-\I,0),L^\gamma)}
	\to0$
as $b\to-\I$.
Hence, it suffices to prove that
\[
	\norm{e^{it\Delta} e^{ib|x|^2} \psi}_{L^\rho ((0,\I),L^\gamma)}
	\to0
\]
as $b\to-\I$.
Let $b<0$.
One verifies that
\[
	e^{it\Delta}( e^{ib|x|^2}\psi)(x)
	=e^{i\frac{b}{1+4bt}|x|^2}(1+4bt)^{-\frac{N}2} \( e^{i\frac{t}{1+4bt}\Delta}\psi \)\(\frac{x}{1+4bt}\).
\]
Hence,
\[
	\Lebn{e^{it\Delta} e^{ib|x|^2}\psi}{\gamma}=\abs{1+4bt}^{-\delta(\gamma)} 
	\Lebn{e^{i\frac{t}{1+4bt}\Delta}\psi}{\gamma}.
\]
We estimate $L_t^\rho((0,1/4|b|))$ and $L_t^\rho((1/4|b|,\I))$ individually.
It holds that
\begin{align*}
	\norm{e^{it\Delta} e^{ib|x|^2}\psi}_{L_t^\rho((0,1/4|b|),L^\gamma)}^\rho={}&
	\int_0^{1/4|b|}\abs{1+4bt}^{-\rho\delta(\gamma)} 
	\Lebn{e^{i\frac{t}{1+4bt}\Delta}\psi}{\gamma}^\rho dt \\
	={}&\int_0^\I (1-4bs)^{\rho\delta(\gamma)-2} \Lebn{e^{is\Delta}\psi}{\gamma}^\rho ds.
\end{align*}
Since $\rho\delta(\gamma)-2<0$, the integrand of the right hand side tends to zero
as $b\to-\I$ for each $s>0$.
Further, Lemma \ref{lem:decay} gives us
\[
	\int_0^\I (1-4bs)^{\rho\delta(\gamma)-2} \Lebn{e^{is\Delta}\psi}{\gamma}^\rho ds
	\le \norm{e^{is\Delta}\psi}_{L^\rho_s((0,\I),L^\gamma)}^\rho \le C \ell(\psi)^\rho<\I.
\]
We then see from Lebesgue's convergence theorem that
\[
	\norm{e^{it\Delta} e^{ib|x|^2}\psi}_{L_t^\rho((0,1/4|b|),L^\gamma)} \to 0
\]
as $b\to-\I$.
Similarly, one deduces that
\[
	\norm{e^{it\Delta} e^{ib|x|^2}\psi}_{L_t^\rho((1/4|b|,\I),L^\gamma)}^\rho=
	\int_{-\I}^{1/4b} (1-4bs)^{\rho\delta(\gamma)-2} \Lebn{e^{is\Delta}\psi}{\gamma}^\rho ds.
\]
Fix a small number $a>0$. Then,
\[
	\int_{-\I}^{-a} (1-4bs)^{\rho\delta(\gamma)-2} \Lebn{e^{is\Delta}\psi}{\gamma}^\rho ds
	\le (1-4ba)^{\rho\delta(\gamma)-2}\norm{e^{is\Delta}\psi}_{L^\rho_s((-\I,0),L^\gamma)}
	\to0
\]
as $b\to-\I$.
Let $\beta=\beta(a)>0$ to be chosen later.
For sufficiently large $|b|$, we have
$-a<\frac1{4b}-\frac{\beta}b^2$.
By H\"older's inequality,
\begin{align*}
	&\int_{-a}^{\frac1{4b}-\frac{\beta}{b^2}}
	(1-4bs)^{\rho\delta(\gamma)-2} \Lebn{e^{is\Delta}\psi}{\gamma}^\rho ds\\
	&{}\le 
	\(\int_{-a}^{\frac1{4b}-\frac{\beta}{b^2}}
	(1-4bs)^{-2} ds\)^{\frac12(1-\frac{\rho\delta(\gamma)}2)}
	\norm{ e^{is \Delta}\psi }_{L^{2/\delta(\gamma)}((-a,0),L^\gamma)}^\rho \\
	&{}\le  c_1 \beta^{-\frac{1}2(1-\frac{\rho\delta(\gamma)}2)}
	\norm{ e^{is \Delta}\psi }_{L^{2/\delta(\gamma)}((-a,0),L^\gamma)}^\rho.
\end{align*}
On the other hand, by Lemma \ref{lem:decay},
\begin{multline*}
	\int_{\frac1{4b}-\frac{\beta}{b^2}}^{\frac1{4b}}
	(1-4bs)^{\rho\delta(\gamma)-2} \Lebn{e^{is\Delta}\psi}{\gamma}^\rho ds\\
	\le 
	C\ell(\psi)^\rho \int_{\frac1{4b}-\frac{\beta}{b^2}}^{\frac1{4b}}
	(1-4bs)^{\rho\delta(\gamma)-2} |s|^{-\rho\delta(\gamma)} ds	\le 
	c_2\beta^{\rho\delta(\gamma)-1}\ell(\psi)^{\rho}.
\end{multline*}
Combining these estimates, we obtain
\begin{multline*}
	\int_{-a}^{\frac1{4b}} (1-4bs)^{\rho\delta(\gamma)-2} \Lebn{e^{is\Delta}\psi}{\gamma}^\rho ds\\
	\le c_1 (\beta^{\frac{\delta(\gamma)}2-\frac1\rho}\norm{ e^{is \Delta}\psi }_{L^{2/\delta(\gamma)}((-a,0),L^\gamma)})^\rho + c_2 (\beta^{\delta(\gamma)-\frac1\rho}
	\ell({\psi}))^\rho.
\end{multline*}
Now, let
	$\beta = \norm{ e^{is \Delta}\psi }_{L^{2/\delta(\gamma)}((-a,0),L^\gamma)}^{2/\delta(\gamma)} \ell(\psi)^{-2/\delta(\gamma)}$.
Then,
\[
	\int_{-a}^{\frac1{4b}} (1-4bs)^{\rho\delta(\gamma)-2} \Lebn{e^{is\Delta}\psi}{\gamma}^\rho ds
	\le C \norm{ e^{is \Delta}\psi }_{L^{2/{\delta(\gamma)}}((-a,0),L^\gamma)}^{2\rho-\frac2{\delta(\gamma)}}
	\ell(\psi)^{\frac2{\delta(\gamma)}-\rho}.
\]
Thus, we reach to the estimate
\[
	\limsup_{b\to-\I} \norm{e^{it\Delta} e^{ib|x|^2}\psi}_{L_t^\rho((0,\I),L^\gamma)}
	\le C\norm{ e^{is \Delta}\psi }_{L^{2/\delta(\gamma)}((-a,0),L^\gamma)}^{2-\frac2{\rho\delta(\gamma)}}
	\ell(\psi)^{\frac2{\rho\delta(\gamma)}-1}.
\]
Recall that $a>0$ is arbitrary.
Since $e^{it\Delta}\psi \in L^{\frac2{\delta(\gamma)}}(\R,L^\gamma)$ by
Strichartz' estimate and
since $2/\delta(\gamma)<\I$ follows from $\gamma>2$, one sees that
\[
	\norm{ e^{is \Delta}\psi }_{L^{2/\delta(\gamma)}((-a,0),L^\gamma)} \to 0
\]
as $a\to0$.
Since $2-\frac{2}{\rho\delta(\gamma)}>0$ by \eqref{eq:acceptable},
we finally obtain
\[
	\limsup_{b\to-\I} \norm{e^{it\Delta} e^{ib|x|^2}\psi}_{L_t^\rho((0,\I),L^\gamma)}=0,
\]
which completes the proof.
\end{proof}
By this proposition, we obtain the following fact on the scattering set $S$.
\begin{corollary}
Suppose $\pst<p<1+\frac4N$.
The scattering set $S$ is an open subset of $\F H^1$ and unbounded in such a sense that
\[
	\sup_{u_0 \in S} \inf_{{\bf a}\in \R^N} \ell(u_0(\cdot-{\bf a}))
	= \sup_{u_0 \in S} \inf_{{\bf a}\in \R^N} \norm{u_0(\cdot-{\bf a})}_{\F H^1}
	=\I.
\]
\end{corollary}
\begin{proof}
Openness immediately follows from Propositions \ref{prop:ScatterCond} and \ref{prop:lpt}.
We prove unboundedness. 
Take a nontrivial radial function $\psi\in \F H^1$.
Then, $\ell(\psi)=\inf_{{\bf a}\in \R^N} \ell(u_0(\cdot-{\bf a}))$.
By Proposition \ref{prop:ODS},
for any constant $C>0$ there exists $b_0=b_0(C)\in \R$ such that
$e^{ib_0|x|^2}(C\psi)\in S$. Then, 
$\norm{e^{ib_0|x|^2}C\psi}_{\F H^1} \ge\ell(e^{ib_0|x|^2}C\psi)=C\ell(\psi)$.
Since $C$ is arbitrary, unboundedness holds.
\end{proof}
\section{Profile decomposition}
This section is devoted to the proof of 
the profile decomposition lemma, which is one of the main tool 
for the proof of Theorem \ref{thm:main1}.
A similar property for a sequence bounded in $\dot{H}^1$
is established in \cite{Ker} (see also \cite{BG}), and applied in \cite{KM}
to the study of NLS equation.
In \cite{DHR,FXC}, this is established for sequences bounded in $H^1$.
\begin{proposition}[profile decomposition lemma]\label{prop:pd}
Suppose $\pst<p<1+\frac4N$.
Let $\{\phi_n\}_n$ be a bounded sequence in $\F H^1$.
There exist a subsequence of $\{\phi_n\}$,
which is denoted again by $\{\phi_n\}$,
and sequences $\{\psi^j\}_j\subset \F H^1$, $\{W_n^j\}_{n,j}\subset \F H^1$,
and $\{\xi_n^j\}_{n,j} \subset \R^N$ such that for every $l\ge1$
\begin{equation}\label{eq:decompose1b}
	\phi_n = \sum_{j=1}^l e^{i\xi_n^j\cdot x} \psi^j + W_n^l,
\end{equation}
in $\F H^1$ and, 
\begin{equation}\label{eq:decompose3b}
	\norm{\phi_n}_{\F \dot{H}^s}^2 -\sum_{j=1}^l \norm{\psi^j}_{\F \dot{H}^s}^2
	- \norm{W^l_n}_{\F \dot{H}^s}^2 \to 0
\end{equation}
as $n\to\I$ for all $s\in [0,1]$.
Furthermore, there exists $J\in [0,\I]$ such that 
$\psi^j\equiv 0$ for all $j\ge1$ if $J=0$,
$\psi^j\not\equiv 0$ for all $j\ge1$ if $J=\I$,
 and
\begin{equation}\label{eq:decompose4b}
	\psi^j\not\equiv0 \text{ for } j\le J , \qquad 
	\psi^j\equiv0 \text{ for }j> J
\end{equation}
holds if $J\in[1,\I)$.
For all $l\ge1$ and $1\le i,j\le J$, $i\neq j$,
\begin{equation}\label{eq:decompose5b}
	\lim_{n\to\I} |\xi_n^i-\xi_n^j| =\I.
\end{equation}
In addition,
\begin{equation}\label{eq:decompose6b}
	\limsup_{n\to\I} \norm{e^{it \Delta}W_n^l}_{L^\rho((0,\I),L^{\gamma})}\to0
\end{equation}
as $l\to\I$.
\end{proposition}
Let us begin the proof of Proposition \ref{prop:pd}
with the following lemma.
\begin{lemma}\label{lem:pd1}
Let $a>0$ and let $\{v_n \}_n \subset \F H^1$ satisfy
\begin{equation}\label{asmp:lem:unifbdd1}
	\limsup_{n\to\I} \norm{v_n}_{\F H^1} \le a <\I.
\end{equation}
If
\begin{equation}\label{asmp:lem:unifbdd2}
	\lim_{n\to\I} \norm{|t|^{\delta(\gamma)} \norm{e^{it\Delta}v_n }_{L^\gamma}}_{L^\I((0,\I))} = A,
\end{equation}
then there exist a subsequence, which denoted again by $\{v_n\}$, 
a function $\psi\in \F H^1$, and sequences
$\{t_n\}_n\subset (0,\I)$, $\{\xi_n\}_n \subset \R^N$,  and
$\{W_n\}_n \subset \F H^1$ such that
\begin{equation}\label{eq:lem:decompose}
	v_n = e^{-i \xi_n\cdot x}e^{-it_n{|x|^2}} \psi + W_n
\end{equation}
with
\begin{equation}\label{eq:lem:weaklimit1}
	e^{i \xi_n \cdot x}e^{it_n{|x|^2}} v_n \rightharpoonup \psi\IN \F H^1
\end{equation}
or equivalently
\begin{equation}\label{eq:lem:weaklimit2}
	e^{i \xi_n \cdot x}e^{it_n{|x|^2}} W_n \rightharpoonup 0 \IN \F H^1
\end{equation}
as $n\to\I$, and
\begin{equation}\label{eq:lem:Pythagorean}
	\norm{v_n}_{\F \dot{H}^s}^2 -\norm{W_n}_{\F \dot{H}^s}^2
	- \norm{\psi}_{\F \dot{H}^s}^2 \to 0
\end{equation}
as $n\to\I$ for all $s\in[0,1]$. Moreover, there exists a
constant $c$ independent of $\{v_n\}_n$, $a$, and $A$ such that
\begin{equation}\label{eq:lem:lowerbound}
	\norm{\psi}_{\F H^1} \ge  
 c a^{1-\frac{\gamma}{(\gamma-2)(1-\delta(\gamma))}}
	A^{\frac{\gamma}{(\gamma-2)(1-\delta(\gamma))}}.
\end{equation}
\end{lemma}
\begin{proof}
Let $\zeta$ be a smooth nonnegative radial function such that $\zeta(x)=1$ for $|x|\le1$
and $\zeta(x)=0$ for $|x|\ge 2$.
Let
\[
	\chi_r(t) = e^{it\Delta} \zeta(x/r) e^{-it\Delta}.
\]
One sees from Lemma \ref{lem:decay} that
\begin{align*}
	\norm{e^{it\Delta} v_n - \chi_r(t) e^{it\Delta}v_n}_{L^\gamma}
	&{} = \norm{e^{it\Delta}(1- \zeta(x/r) )v_n}_{L^\gamma} \\
	&{} \le c_0 |t|^{-\delta(\gamma)} \Lebn{(1- \zeta(x/r) )v_n}2^{1-\delta(\gamma)}
	\Lebn{x v_n}2^{\delta(\gamma)} \\
	&{} \le c_0 |t|^{-\delta(\gamma)} r^{-(1-\delta(\gamma))}\Lebn{x v_n}2\\
	&{} \le c_0 |t|^{-\delta(\gamma)}a r^{-(1-\delta(\gamma))}.
\end{align*}
We chose $r=(2c_0 a A^{-1})^{1/(1-\delta(\gamma))}$.
This is possible because $\delta(\gamma)<1$.
We then deduce from \eqref{asmp:lem:unifbdd2} that
\begin{align*}
	&\sup_{t>0}|t|^{\delta(\gamma)}\norm{ \chi_r e^{it\Delta}v_n}_{L^\gamma}\\
	&{}\ge \sup_{t>0}\(|t|^{\delta(\gamma)}\norm{e^{it\Delta} v_n}_{L^\gamma}- 
	|t|^{\delta(\gamma)}\norm{e^{it\Delta} v_n - \chi_r e^{it\Delta}v_n}_{L^\gamma}\) \\
	&{}\ge \sup_{t>0}\(|t|^{\delta(\gamma)}\norm{e^{it\Delta} v_n}_{L^\gamma}- \frac12 A \) 
	\ge \frac34 A - \frac12 A = \frac{A}4
\end{align*}
for sufficiently large $n$.
By H\"older's inequality, it holds that
\begin{align*}
	|t|^{\delta(\gamma)}\norm{ \chi_r e^{it\Delta}v_n}_{L^\gamma}
	&{}\le |t|^{\delta(\gamma)}\norm{ \chi_r e^{it\Delta}v_n}_{L^2}^{\frac2\gamma}
	\norm{ \chi_r e^{it\Delta}v_n}_{L^\I}^{1-\frac2\gamma} \\
	&{}\le a^{\frac2\gamma} 
	\( |t|^{\frac{N}2}
	\norm{ \chi_r e^{it\Delta}v_n}_{L^\I}\)^{1-\frac2\gamma}
\end{align*}
for any $t\neq0$. Combining above estimates, we obtain
\[
	\sup_{t>0} |t|^{\frac{N}2}
	\norm{ \chi_r (t)e^{it\Delta}v_n}_{L^\I}
	\ge\( \frac{A}{4a^{2/\gamma}}\)^{\frac\gamma{\gamma-2}}.
\]
Hence, there exist sequences $\{t_n\}_n \subset (0,\I)$
 and $\{\xi_n\}_n \subset \R^N$ such that
\[
	|4 t_n|^{-\frac{N}2}
	\abs{\chi_r\(\frac1{4 t_n}\) e^{i\frac1{4t_n}\Delta}v_n} \(-\frac{\xi_n}{2t_n}\)
	\ge \( \frac{A}{8a^{2/\gamma}}\)^{\frac\gamma{\gamma-2}}.
\]
Set
\[
	w_n(x) :=  e^{i{\xi_n \cdot x}} e^{it_n{|x|^2}} v_n(x) .
\]
Since $\{w_n\}$ is bounded in $\F H^1$,
we can extract a subsequence,
denoted again by $\{w_n\}$, which converges weakly in $\F H^1$.
Let the weak limit $\psi \in \F H^1$.
Then, \eqref{eq:lem:weaklimit1} holds.
By the definition of $\chi_r$ and
the integral representation of $e^{i\frac1{4t_n}\Delta}$, one deduces that
\begin{align*}
	&(4t_n)^{-\frac{N}2}
	\(\chi_r \(\frac1{4t_n}\)e^{i\frac1{4t_n}\Delta}v_n\)\(-\frac{\xi_n}{2t_n}\)\\
	={}& (4\pi i)^{-\frac{N}2} \int_{\R^N} e^{it_n{|\xi_n/2t_n+y |^2}} \zeta\(\frac{y}r\) v_n(y) \\
	={}& (4\pi i)^{-\frac{N}2} e^{i\frac{|\xi_n|^2}{4t_n}}\int_{\R^N} \zeta\(\frac{y}r\) w_n(y) dy .
\end{align*}
By extracting a subsequence, $e^{i\frac{|\xi_n|^2}{4t_n}}$ converges.
Denote the limit by $e^{i\theta}$. Then, 
\begin{align*}
	(4t_n)^{-\frac{N}2}
	\(\chi_r \(\frac1{4t_n}\)e^{i\frac1{t_n}\Delta}v_n\)\(-\frac{\xi_n}{2t_n}\)
	\to (4\pi i)^{-\frac{N}2} e^{i\theta}\int_{\R^N} \zeta\(\frac{y}r\) \psi (y) dy.
\end{align*}
Therefore, for $n$ large enough,
\[
	|4t_n|^{-\frac{N}2}
	\abs{\chi_r\(\frac1{4t_n}\) e^{i\frac1{4t_n}\Delta}v_n} \(-\frac{\xi_n}{2t_n}\)
	\le c_1 r^{\frac{N}2} \norm{\psi}_{L^2}. 
\]
Thus, we conclude that
\[
	\norm{\psi}_{L^2} \ge 
	 C A^{\frac{\g}{\g-2}} a^{-\frac{2}{\g-2}} r^{ -\frac{N}2}
	= c a^{1-\frac{\gamma}{(\gamma-2)(1-\delta(\gamma))}}
	A^{\frac{\gamma}{(\gamma-2)(1-\delta(\gamma))}},
\]
which yields \eqref{eq:lem:lowerbound} since $\norm{\psi}_{L^2}
\le \norm{\psi}_{\F H^1}$.

Set
\begin{equation}\label{def:lem:Wn}
	W_n:= v_n - e^{-i\xi_n\cdot x}e^{-it_n{|x|^2}} \psi.
\end{equation}
Then, \eqref{eq:lem:weaklimit2}
immediately follows from \eqref{eq:lem:weaklimit1}.
Further, for $s \in [0,1]$,
\begin{align*}
	\norm{W_n}_{\F \dot{H}^s}^2
	={}& \norm{v_n}_{\F \dot{H}^s}^2 + \norm{\psi}_{\F \dot{H}^s}^2-
	2\Re \Jbr{v_n, e^{-i\xi_n\cdot x}e^{-it_n{|x|^2}} \psi }_{\F \dot{H}^s} \\
	={}& \norm{v_n}_{\F \dot{H}^s}^2 + \norm{\psi}_{\F \dot{H}^s}^2-
	2\Re \Jbr{w_n, \psi }_{\F \dot{H}^s}.
\end{align*}
By means of \eqref{eq:lem:weaklimit1},
 $\Re \Jbr{w_n, \psi }_{\F \dot{H}^s} \to \norm{\psi}_{\F \dot{H}^s}^2$ 
as $n\to \I$. This completes the proof of  \eqref{eq:lem:Pythagorean}.
\end{proof}
\begin{lemma}\label{lem:pd2}
Let $\{\tau_n\}_n \subset \R$ and $\{\xi_n \}_n \subset \R^N$ satisfy
\begin{equation}\label{eq:lem2:divergence}
	|\tau_n| + |\xi_n| \to \I
\end{equation}
as $n\to\I$. Then, if follows for all $\psi\in \F H^1$ that
\begin{equation}\label{eq:lem2:weaklimit}
	e^{i\tau_n |x|^2} e^{ix\cdot\xi_n} \psi \rightharpoonup 0 \IN \F H^1
\end{equation}
as $n\to\I$.
Conversely, if $\{z_n\}_N \subset \F H^1$ satisfy
\[
	z_n \rightharpoonup 0 \IN \F H^1, \quad e^{i\tau_n |x|^2} e^{ix\cdot\xi_n}z_n 
	\rightharpoonup \psi \IN \F H^1
\] 
as $n\to\I$ for some $\{\tau_n\}_n \subset (0,\I)$, $\{\xi_n \}_n \subset \R^N$ and $\psi \in \F H^1$, $\psi\neq0$,
then \eqref{eq:lem2:divergence} holds.
\end{lemma}
\begin{proof}
By the Fourier transform, it holds that, as $n\to\I$,
\[
	e^{i\tau_n |x|^2} e^{ix\cdot\xi_n} \psi \rightharpoonup 0 \IN \F H^1
\]
is equivalent to
\[
	e^{-i\tau_n \Delta} (\F\psi)(\cdot + \xi_n) \rightharpoonup 0 \IN H^1.
\]
The lemma now follows from Lemma 5.3 of \cite{FXC}.
\end{proof}
The next lemma is our first decomposition result, which is similar to
those for $H^1$-bounded sequences obtained in \cite{DHR,FXC}.
This decomposition involves quadratic oscillation.
\begin{lemma}\label{lem:pd3}
Under assumption of Proposition \ref{prop:pd},
there exist a subsequence of $\{\phi_n\}$,
which is denoted again by $\{\phi_n\}$,
and sequences $\{\psi^j\}_j\subset \F H^1$, $\{W_n^j\}_{n,j}\subset \F H^1$,
$\{t_n^j\}_{n,j}\subset (-\I,0)$, $\{\overline{t}^j\}_j \subset [-\I,0]$,
and $\{\xi_n^j\}_{n,j} \subset \R^N$ such that for every $l\ge1$
\begin{equation}\label{eq:decompose1}
	\phi_n = \sum_{j=1}^l e^{i\xi_n^j\cdot x} e^{it_n^j |x|^2 } \psi^j + W_n^l,
\end{equation}
in $\F H^1$ and, as $n\to\I$,
\begin{equation}\label{eq:decompose2}
	t_n^j \to \overline{t}^j,
\end{equation}
\begin{equation}\label{eq:decompose3}
	\norm{\phi_n}_{\F \dot{H}^s}^2 -\sum_{j=1}^l \norm{\psi^j}_{\F \dot{H}^s}^2
	- \norm{W^j_n}_{\F \dot{H}^s}^2 \to 0
\end{equation}
for all $s\in [0,1]$.
Furthermore, there exists $J\in \N \cup \{\I\}$ such that 
$\psi^j\equiv 0$ for all $j\ge1$ if $J=0$,
$\psi^j\not\equiv 0$ for all $j\ge1$ if $J=\I$, and
\begin{equation}\label{eq:decompose4}
	\psi^j\not\equiv0 \text{ for } j\le J , \qquad 
	\psi^j\equiv0 \text{ for }j> J
\end{equation}
if $J\in[1,\I)$.
For all $l\ge1$ and $1\le i,j\le J$, $i\neq j$,
\begin{equation}\label{eq:decompose5}
	\lim_{n\to\I} |t_n^i-t_n^i| + |x_n^i-x_n^j| =\I.
\end{equation}
In addition,
\begin{equation}\label{eq:decompose6}
	\limsup_{n\to\I} \norm{e^{it \Delta}W_n^l}_{L^\rho((0,\I),L^{\gamma})}\to0
\end{equation}
as $l\to\I$.
\end{lemma}
\begin{remark}
As shown in \cite{DHR,FXC},
a $H^1$-bounded sequence $\{\varphi_n\}_n$ is decomposed, up to a subsequence, into
a form
\[
	\varphi_n = \sum_{j=1}^l e^{-it_n^j\Delta} \psi^j(x-x_n^j) + W_n^j.
\]
Our decomposition \eqref{eq:decompose1} is similar to 
the Fourier transform of this in such a sense that the both give
profiles of the same form.
However, it is not cleat whether the estimate \eqref{eq:decompose6} follows 
from this correspondence.
\end{remark}
\begin{proof}
Let
\[
	a = \limsup_{n\to\I} \norm{\phi_n}_{\F H^1}>0.
\]
Let $W_n^0=\phi_n$. We construct by induction on $l$ the various sequences so
that for every $1\le j \le l$,
\begin{equation}\label{eq:prop:property1}
	W^{j-1}_n = e^{i\xi_n^j \cdot x}e^{i t_n^j{|x|^2} } \psi^j + W^j_n
\end{equation}
for $n\ge1$ and
\begin{equation}\label{eq:prop:property2}
	t_n^j \to \overline{t}^j \in [-\I,0],
\end{equation}
\begin{equation}\label{eq:prop:property3}
	e^{-i\xi_n^j \cdot x}e^{-i t_n^j{|x|^2}} W_n^{j-1} \rightharpoonup
	\psi^j,\quad
	e^{-i\xi_n^j \cdot x}e^{-i t_n^j{|x|^2}} W_n^{j} \rightharpoonup 0
\end{equation}
in $\F H^1$ as $n\to\I$,
\begin{equation}\label{eq:prop:property4}
	\norm{W^{j-1}_n}_{\F \dot{H}^s}^2 - \norm{\psi^j}_{\F \dot{H}^s}^2
	-\norm{W_n^j}_{\F \dot{H}^s}^2 \to 0
\end{equation}
as $n\to \I$ for all $s\in[0,1]$, and
\begin{equation}\label{eq:prop:property5}
	\norm{|t|^{\delta(\gamma)} \norm{e^{it\Delta}W^{j-1}_n}_{L^\gamma}}_{L^\I((0,\I))}
	\to A_j,
\end{equation}
\begin{equation}\label{eq:prop:property6}
	\norm{\psi^j}_{\F H^1} \ge \nu(a)
	A_j^{\frac{\gamma}{(\gamma-2)(1-\delta(\gamma))}}.
\end{equation}

For $l=1$, we set
\[
	A_1 = \limsup_{n\to\I} \norm{ |t|^{\delta(\gamma)}
\norm{e^{it\Delta}\phi_n}_{L^\gamma}}_{L^\I((0,\I))}.
\]
We can extract a subsequence so that 
\[
\lim_{n\to\I} \norm{ |t|^{\delta(\gamma)}
\norm{e^{it\Delta}\phi_n}_{L^\gamma}}_{L^\I((0,\I))}=A_1.
\]
We apply Lemma \ref{lem:pd1} with $v_n= W^0_n := \phi_n$.
Then, we obtain $\psi^1$, $\{ t_n^1 \}_n$, $\{ \xi_n \}_n$, and $\{W_n^1\} \subset 
\F H^1$ which ensure \eqref{eq:prop:property1},
\eqref{eq:prop:property3},
\eqref{eq:prop:property4},
\eqref{eq:prop:property5}, and
\eqref{eq:prop:property6}.
Moreover, by extracting a subsequence if necessary, we claim that
$t_n^1\to \overline{t}^1\in [-\I,0]$.
This is \eqref{eq:prop:property2}.

Fix $l\ge2$ and suppose that $\{t_n^j\}_n$, $\overline{t}^j$, $\{\xi_n^j\}_n$,
$\psi^j$, and $\{W_n^j\}_n$ are successfully constructed for all $j\le l-1$.
Set 
\[
	A_l = \limsup_{n\to\I} \norm{ |t|^{\delta(\gamma)}
\norm{e^{it\Delta}W_n^{l-1}}_{L^\gamma}}_{L^\I((0,\I))}.
\]
By extracting a subsequence,
we can replace $\limsup$ by $\lim$, which implies \eqref{eq:prop:property5} holds.
We now apply Lemma \ref{lem:pd1} with $v_n = W^{l-1}_n$.
Then, as in the case $l=1$, we obtain $\{t_n^l\}_n$,
$\overline{t}^l$, $\{\xi_n^l\}$, $\psi^l$, and $\{W^l_{n}\}_n$ so that 
\eqref{eq:prop:property1}, \eqref{eq:prop:property2}, 
\eqref{eq:prop:property3},
and \eqref{eq:prop:property4} hold with $j=l$.
Summing up the equation \eqref{eq:prop:property4} with respect to $j$ yields
\[
	\sum_{j=1}^l \norm{\psi^j}_{\F H^1}^2 \le a^2,
\]
showing in particular $\norm{\psi^l}_{\F H^1}\le a$.
By \eqref{eq:lem:lowerbound}, we have
 \eqref{eq:prop:property6} for $j=l$.

We shall check the sequences
we have constructed satisfy the desired properties.
The properties \eqref{eq:decompose1}, 
\eqref{eq:decompose2}, and  \eqref{eq:decompose3} follow
immediately from \eqref{eq:prop:property1},
\eqref{eq:prop:property2}, 
\eqref{eq:prop:property4}, and $W^0_n = \phi_n$.
The estimates \eqref{eq:prop:property4} and
\eqref{eq:prop:property6} infer that
\[
	\sum_{j=1}^\I A_j^{\frac{2\gamma}{(\gamma-2)(1-\delta(\gamma))}}	
	\le C(a) \sum_{j=1}^\I \norm{\psi^j}_{\F H^1}^2 
	<\I,
\]
showing that $A_l \to 0$ as $l\to \I$. Now,
it holds for any $T>0$ that
\begin{align*}
	\norm{e^{it\Delta}W_n^l}_{L^\rho((0,\I),L^\gamma)}
	\le{}& T^{\frac1\rho-\frac1\eta}\norm{e^{it\Delta}W_n^l}_{L^\eta((0,T),L^\gamma)}\\
	&{}+ T^{\frac1\rho-\delta(\gamma)} \norm{|t|^{\delta(\gamma)} \norm{e^{it\Delta}W_n^l}_{L^\gamma}}_{L^\I((T,\I))},
\end{align*}
where $\eta=2/\delta(\gamma)$.
Minimizing the right hand side with respect to $T$ and applying Strichartz' estimate,
we obtain
\[
	\norm{e^{it\Delta}W_n^l}_{L^\rho((0,\I),L^\gamma)}
	\le C \norm{W_n^l}_{L^2}^{2-\frac2{\rho\delta(\gamma)}}
	\norm{|t|^{\delta(\gamma)} \norm{e^{it\Delta}W_n^l}_{L^\gamma}}_{L^\I((0,\I))}^{
	\frac2{\rho\delta(\gamma)}-1}.
\]
Thus, $A_l\to0$ as $l\to\I$ implies
\[
	\lim_{n\to\I }\norm{e^{it\Delta}W_n^{l}}_{L^\rho((0,\I),L^\gamma)}
	\to 0
\]
as $l\to\I$, which is \eqref{eq:decompose6}.
Let us next show the existence of $J\in[0,\I]$.
If $\psi^{j_0}=0$ for some $j_0\ge1$ in the above procedure, then $A_{j_0}=0$ holds by 
\eqref{eq:prop:property6}.
All the properties (except for \eqref{eq:decompose5}) are 
then satisfied with the choice $\psi^l=0$ and $W^l_n=W^{j_0}_n$ for all $l\ge {j_0}$.
Hence, we let $J=j_0-1$ in this case.
If such $j_0$ does not exist then we let $J=\I$.
Let us proceed to the proof of \eqref{eq:decompose5}.
Assume $J\ge2$ and
take $2\le l\le J$. 
We shall show by induction on $k$ that 
\begin{equation}\label{eq:prop:induction}
	|t_n^l-t_n^{l-k}| + |\xi^l_n-\xi^{l-k}_n| \to \I
\end{equation}
as $n\to\I$ for all $1\le k \le l-1$.
First consider the case $k=1$.
By \eqref{eq:prop:property3} with $j=l-1,l$, it follows that
\[
	e^{-i\xi_n^{l-1} \cdot x}e^{-i t_n^{l-1}{|x|^2}} W_n^{l-1} \rightharpoonup 0,
	\quad e^{-i\xi_n^l \cdot x}e^{-i t_n^l{|x|^2}} W_n^{l-1} \rightharpoonup
	\psi^l
\]
in $\F H^1$ as $n\to \I$. 
Apply Lemma \ref{lem:pd2} with $z_n=e^{-i\xi_n^{l-1} \cdot x}e^{-i t_n^{l-1}{|x|^2}} W_n^{l-1}$
to yield
\[
	|t_n^l-t_n^{l-1}| + |\xi^l_n-\xi^{l-1}_n| \to \I
\]
as $n\to\I$.
The proof is done when $J=2$. Hence we let $J\ge3$ and $l\ge3$.
Suppose that \eqref{eq:prop:induction} is true for all $1 \le k\le k_0-1$,
where $k_0\in[2,l-1]$.
Then, by \eqref{eq:prop:property1},
\[
	W^{l-1}_n - W^{l-k_0}_n = \sum_{j=l-k_0+1}^{l-1} e^{i\xi_n^{j} \cdot x}e^{it_n^{j}{|x|^2}}
	\psi^j.
\]
This implies
\begin{multline*}
	e^{-i\xi_n^{l} \cdot x}e^{-i t_n^{l}{|x|^2}} W^{l-k_0}_n  \\
	=e^{-i\xi_n^{l} \cdot x}e^{-i t_n^{l}{|x|^2}} W^{l-1}_n -  \sum_{j=l-k_0+1}^{l-1} e^{i(\xi_n^{j} -\xi_n^{l})\cdot x}e^{i(t_n^{j}-t_n^{l}){|x|^2}}
	\psi^j.
\end{multline*}
One then sees from \eqref{eq:prop:property3} with $j=l$
that $
	e^{-i\xi_n^{l} \cdot x}e^{-i t_n^{l}{|x|^2}} W^{l-1}_n \rightharpoonup \psi^l
$
in $\F H^1$ as $n\to\I$.
Moreover, it follows
from the former part of Lemma \ref{lem:pd2} and assumption of induction
that
\[
	\sum_{j=l-k_0+1}^{l-1} e^{i(\xi_n^{j} -\xi_n^{l})\cdot x}e^{i(t_n^{j}-t_n^{l}){|x|^2}}
	\psi^j \rightharpoonup 0
\]
in $\F H^1$ as $n\to\I$.
Hence, as $n\to\I$,
\[
	e^{-i\xi_n^{l} \cdot x}e^{-i t_n^{l}{|x|^2}} W^{l-k_0}_n \rightharpoonup \psi^l.
\]
On the other hand, by \eqref{eq:prop:property3} with $j=l-k_0$,
we have
\[
	e^{-i\xi_n^{l-k_0} \cdot x}e^{-i t_n^{l-k_0}{|x|^2}} W^{l-k_0}_n \rightharpoonup 0
\]
in $\F H^1$ as $n\to\I$.
The latter part of Lemma \ref{lem:pd2} shows
$|t_n^l-t_n^{l-k_0}| + |\xi^l_n-\xi^{l-k_0}_n| \to \I$
as $n\to\I$, which is \eqref{eq:prop:induction} with $l=k_0$.
Hence, \eqref{eq:prop:induction} is true for all $1\le k\le l-1$ and $2\le l  \le J$,
which completes the proof of \eqref{eq:decompose5}.
\end{proof}
We finally prove Proposition \ref{prop:pd}.
To do so, it is essential to show that we can let $\overline{t}^j>-\I$ for all $j$
in Lemma \ref{lem:pd3}.
Heart of matter is that if $t_n^j\to-\I$ as $n\to\I$ then $e^{it_n^j|x|^2}\psi^j$
becomes a remainder term by means of Proposition \ref{prop:ODS}.
\begin{proof}[Proof of Proposition \ref{prop:pd}]
We first apply Lemma \ref{lem:pd3} to obtain
\[
	\phi_n = \sum_{j=1}^l e^{i\xi_n^j\cdot x} e^{it_n^j |x|^2 } \varphi^j + w_n^l.
\]
Set $L=\{l\in [1,J]\ | \ \overline{t}^{l} >-\I \}$.
If $L=\emptyset$ then the result follows with $\psi^j=0$ for $j\ge1$
and $W^{l}_n=\phi_n=\sum_{j=1}^l e^{i\xi_n^j\cdot x} e^{it_n^j |x|^2 } \varphi^j + w_n^l$
for $l\ge1$.
Indeed, \eqref{eq:decompose1b} and \eqref{eq:decompose3b}
are obvious, $J=0$, and
we see from Proposition \ref{prop:ODS} that
\[
	\limsup_{n\to\I} \norm{e^{it \Delta}W_n^l}_{L^\rho((0,\I),L^{\gamma})}
	= \limsup_{n\to\I} \norm{e^{it \Delta}w_n^l}_{L^\rho((0,\I),L^{\gamma})}
	\to 0 
\]
as $l\to \I$, which is \eqref{eq:decompose6b}.

Consider the case $L\neq \emptyset$.
We may first assume $L$ is not a finite set.
We number the elements of $L$ as $L=\{j_1,j_2,j_3,\dots\} $ 
in such a way that $j_k<j_l$ as long as $k<l$.
For each $ l\ge1$, we define
\[
	\psi^l = e^{i\overline{t}^{j_l} |x|^2 } \varphi^{j_l}
\]
and
\[
	W^l_n = \sum_{k=1}^l e^{i\xi_n^{j_k}\cdot x} (e^{it_n^{j_k} |x|^2 } \varphi^{j_k}-\psi^k) 
	+ \sum_{j\le j_l, j\not\in L}e^{i\xi_n^j\cdot x} e^{it_n^j |x|^2 } \varphi^j + w_n^{j_l}.
\]
Then, \eqref{eq:decompose1b} holds.
Remark that $e^{i\xi_n^{j_l}\cdot x} (e^{it_n^{j_l} |x|^2 } \varphi^j-\psi^l)\to 0$
strongly in $\F H^1$ as $n\to\I$.
Therefore, by Lemma \ref{lem:decay}, we have
\[
	\limsup_{n\to\I} \norm{e^{it \Delta}W_n^l}_{L^\rho((0,\I),L^{\gamma})}
	= \limsup_{n\to\I} \norm{e^{it \Delta}w_n^{j_l}}_{L^\rho((0,\I),L^{\gamma})}
	\to 0 
\]
as $l\to\I$. Thus, \eqref{eq:decompose6b} holds.
Further, as in the proof of Lemma \ref{lem:pd3}, one sees that
\[
	\norm{W^l_n}_{\F \dot{H}^s}^2
	- \sum_{j\le j_l, j\not\in L} \norm{\varphi^j}_{\F \dot{H}^s}^2
	- \norm{w^{j_l}_n}_{\F \dot{H}^s}^2 \to 0
\]
as $n\to\I$ for any $s\in[0,1]$. Therefore,
\begin{align*}
	&{}\norm{\phi_n}_{\F \dot{H}^s}^2 -\sum_{j=1}^l \norm{\psi^j}_{\F \dot{H}^s}^2
	- \norm{W^l_n}_{\F \dot{H}^s}^2 \\
	={}& \(\norm{\phi_n}_{\F \dot{H}^s}^2 - \sum_{j=1}^{j_l} \norm{\varphi^j}_{\F \dot{H}^s}^2
	- \norm{w^{j_l}_n}_{\F \dot{H}^s}^2\) \\
	&{}- \(\norm{W^l_n}_{\F \dot{H}^s}^2
	- \sum_{j\le j_l, j\not\in L} \norm{\varphi^j}_{\F \dot{H}^s}^2
	- \norm{w^{j_l}_n}_{\F \dot{H}^s}^2 \)\to0
\end{align*}
as $l\to\I$, showing \eqref{eq:decompose3b}.
The property
\eqref{eq:decompose5b} is an immediate consequence of \eqref{eq:decompose5}
and the fact that $\overline{t}_j>-\I$ for $j\in L$.

If $L$ is a finite set then it suffices to define
 $\psi^l$ and $W^l_n$ for $l> \sharp L$ as follows;
$\psi^l=0$ and $W^l_n =w^{\max L}_n$. Then, \eqref{eq:decompose4b} is trivial.
Remark that 
\[
	\limsup_{n\to\I }\norm{e^{it \Delta} w^{\max L}_n}_{L^\rho((0,\I),L^{\gamma})}
	= \limsup_{n\to\I }\norm{e^{it \Delta} w^{l}_n}_{L^\rho((0,\I),L^{\gamma})}
	\to 0
\]
as $l\to\I$ by Proposition \ref{prop:ODS}, that is,
$\limsup_{n\to\I }\norm{e^{it \Delta} w^{\max L}_n}_{L^\rho((0,\I),L^{\gamma})}
	= 0$. Hence \eqref{eq:decompose6b} holds.
\end{proof}
\section{Proof of Theorem \ref{thm:main1}}
Let us define
\[
	\ell_{c}
	:= \inf\{ \ell (f) \ |\ f \in \F H^1 \setminus S_+ \}.
\]
By Corollary \ref{cor:sds},
there exists a constant $\eta_1>0$ such that $\ell_c \ge \eta_1$.
On the other hand, since $Q \in \F H^1 \setminus S_+$, we have $\ell_c \le 
\ell (Q)<\I$, where $Q$ is a ground state.
Thus, $\ell_c \in (0,\I)$.

The main step of the proof is the following.
\begin{lemma}\label{lem:oneprofile}
Let $\{u_{0,n}\}_n \subset \F H^1$ be a sequence satisfying
$u_{0,n} \not\in S_+$ and $\ell(u_{0,n})\le \ell_c + \frac1n$.
Assume $\Lebn{u_{0,n}}2=1$.
Then, there exist a subsequence of $\{u_{0,n}\}_n$,
which is denoted again by $\{u_{0,n}\}_n$,
a function $\psi \in  \F H^1$ with $\Lebn{\psi}2=1$ and $\ell(\psi)=\ell_c$,
and sequences $\{W_n\}_{n}\subset \F H^1$
and $\{\xi_n\}_{n} \subset \R^N$ such that
\begin{equation}\label{eq:cc1}
	u_{0,n} = e^{i\xi_n\cdot x} \psi + W_n
\end{equation}
and
\begin{equation}\label{eq:cc2}
	\norm{W_n}_{\F H^1} \to 0
\end{equation}
as $n\to\I$.
\end{lemma}
\begin{proof}
It holds by assumption $\Lebn{u_{0,n}}2=1$ that
\begin{equation}\label{eq:critpf0}
	\limsup_{n\to\I} \Lebn{xu_{0,n}}2^{\frac2{p-1} - \frac{N}2} \le
	\ell_{c}.
\end{equation}
Hence, $u_{0,n}$ is uniformly bounded in $\F H^1$.
By profile decomposition (Proposition \ref{prop:pd}), there exist 
a subsequence of $\{u_{0,n}\}$, which is denoted again by $\{u_{0,n}\}$,
and sequences $\{\psi^j\}_j\subset \F H^1$, $\{W_n^j\}_{n,j}\subset \F H^1$,
and $\{\xi_n^j\}_{n,j} \subset \R^N$ such that for every $l\ge1$
\begin{equation}\label{eq:critpf1}
	u_{0,n} = \sum_{j=1}^l e^{i\xi_n^j\cdot x} \psi^j + W_n^l,
\end{equation}
in $\F H^1$ and, as $n\to\I$,
\begin{equation}\label{eq:critpf3}
	\norm{u_{0,n}}_{\F \dot{H}^s}^2 -\sum_{j=1}^l \norm{\psi^j}_{\F \dot{H}^s}^2
	- \norm{W^j_n}_{\F \dot{H}^s}^2 \to 0
\end{equation}
for all $s\in [0,1]$.
Furthermore, there also exists $J\in \N \cup \{\I\}$ such that 
\eqref{eq:decompose4b} holds.
For all $l\ge1$ and $1\le i,j<J$, $i\neq j$,
\begin{equation}\label{eq:critpf4}
	\lim_{n\to\I}  |\xi_n^i-\xi_n^j| =\I.
\end{equation}
In addition,
\begin{equation}\label{eq:critpf5}
	\limsup_{n\to\I} \norm{e^{it \Delta}W_n^l}_{L^\rho((0,\I),L^{\gamma})}\to0
\end{equation}
as $l\to\I$.
Property \eqref{eq:critpf3} with $s=0$ yields
\[
	1 -\sum_{j=1}^l \norm{\psi^j}_{L^2}^2
	- \norm{W^j_n}_{L^2}^2 \to 0.
\]
In particular, $\sum_{j=1}^l \norm{\psi^j}_{L^2}^2\le1$.
Since $l$ is arbitrary, we have
\begin{equation}\label{eq:critpf6}
	\sum_{j=1}^\I \norm{\psi^j}_{L^2}^2\le1.
\end{equation}
Repeating this argument with $s=1$, we deduce from \eqref{eq:critpf0} that
\begin{equation}\label{eq:critpf7}
\begin{aligned}
	\sum_{j=1}^\I \norm{x\psi^j}_{L^2}^2
	\le{}& \limsup_{n\to\I} \norm{xu_{0,n}}_{L^2}^2 \\
	\le{}& \ell_{c}^{2(\frac2{p-1} - \frac{N}2)^{-1}}.
\end{aligned}
\end{equation}
Thus, $\ell(\psi^j)\le \ell_{c}$. 

We now claim that there exists $j_0$ such that $\ell(\psi^{j_0})=\ell_c$.
This claim completes the proof.
Indeed, if such $j_0$ exists then
the inequality
$$\Lebn{x\psi^{j_0}}2\le\ell_c^{(\frac2{p-1} - \frac{N}2)^{-1}},$$
which follows from \eqref{eq:critpf7},
yields $\Lebn{\psi^{j_0}}2\ge1$.
On the other hand, \eqref{eq:critpf6} gives $\Lebn{\psi^{j_0}}2\le1$.
Therefore, $\Lebn{\psi^{j_0}}2=1$.
Further, it then follows from \eqref{eq:critpf6} that $\psi^j\equiv0$ for all $j\neq j_0$.
We hence obtain \eqref{eq:cc1} with $\psi = \psi^{j_0}$ and $W_n=W_n^{j_0}$.
The property \eqref{eq:cc2} immediately follows from \eqref{eq:critpf3}
with $s=0,1$.
Thus, the lemma is proven.

Let us show the claim.
We assume for contradiction that 
$\ell(\psi^j)<\ell_c$ for all $j$.
Then, $\psi^j \in S_+$ by definition of $\ell_{c}$.
Let $V_j$ be a solution of \eqref{eq:NLS} with $V_j|_{t=0}=\psi^j$
and let
\[
	\widetilde{u}_n^l(t,x) := \sum_{j=1}^l V_j(t, x-t\xi^j_n) e^{i\xi_n^j\cdot x}
	e^{i\frac{t}2 |\xi_n^j|^2}.
\]
It follows that
\begin{align*}
	(i\d_t + \Delta) \widetilde{u}^{l}_n (t)
	={}& \sum_{j=1}^l (i\d_t + \Delta)(V_j(t, x-t\xi^j_n) e^{i\xi_n^j\cdot x}
	e^{i\frac{t}2 |\xi_n^j|^2}) \\
	={}& -\sum_{j=1}^l (|V_j|^{p-1} V_j )(t, x-t\xi^j_n) e^{i\xi_n^j\cdot x}
	e^{i\frac{t}2 |\xi_n^j|^2}.
\end{align*}
We also let
\[
	\widetilde{e}^l_n := (i\d_t + \Delta) \widetilde{u}_n^l
	+ |\widetilde{u}_n^l|^{p-1} \widetilde{u}_n^l.
\]

We shall choose $A>0$ independent of $l$ and $n$
and show, for all $l\gg1$, there exists $n_1=n_1(l)$ such that 
\begin{equation}\label{eq:critpf9}
	\norm{\widetilde{u}_n^l}_{L^\rho((0,\I),L^{\gamma})} \le A
\end{equation}
holds true for $n\ge n_1(l)$.
By Proposition \ref{prop:ScatterCond}, $V_j(0)=\psi^j \in S_+$ implies
 $\norm{V_j}_{L^\rho((0,\I),L^\gamma)}<\I$.
Let $\eps>0$ to be chosen later.
In light of \eqref{eq:critpf6} and \eqref{eq:critpf7},
there exists $j_1>0$ such that
\[
	\(\sum_{j=j_1+1}^\I \norm{\psi^j}_{\F H^1}^2\)^{1/2} \le \eps.
\] 
Let $\widetilde{v}_n^{l} = \widetilde{u}_n^l - \widetilde{u}_n^{j_1}$ for $l>j_1$.
By \eqref{eq:critpf4}, it hold for any fixed $l>j_1$ that
\begin{align*}
	\norm{\widetilde{v}_n^{l}(0)}_{\F H^1}^2
	={}& \Lebn{\sum_{j=j_1+1}^l \psi^j e^{i\xi_n^j \cdot x}}2^2
	+ \Lebn{x\sum_{j=j_1+1}^l \psi^j e^{i\xi_n^j \cdot x}}2^2\\
	\to{}& \sum_{j=j_1+1}^l\norm{ \psi^j}_{\F H^1}^2
	\le \eps^2 
\end{align*}
as $n\to\I$.
Hence, for each $l>j_1$, there exists $n_{1,1}=n_{1,1}(l)$ such that
\begin{equation}\label{eq:critpfm1}
	\norm{\widetilde{v}_n^{l}(0)}_{\F H^1}
	\le 2\eps
\end{equation}
for all $n>n_{1,1}$. 
Here, $\widetilde{v}_n^l$ solves
\begin{equation}\label{eq:critpfm2}
	\widetilde{v}_n^l(t)
	=e^{it\Delta} \widetilde{v}_n^l(0) +i \int_0^t
	e^{i(t-s)\Delta}(|\widetilde{v}_n^l|^{p-1}\widetilde{v}_n^l-\widetilde{f}_n^l)(s)ds,
\end{equation}
where
\[
	\widetilde{f}_n^l := 
	|\widetilde{v}_n^l|^{p-1}\widetilde{v}_n^l
	-\sum_{j=j_1+1}^l (|V_j|^{p-1} V_j )(t, x-t\xi^j_n) e^{i\xi_n^j\cdot x}
	e^{i\frac{t}2 |\xi_n^j|^2}.
\]
We apply non-admissible Strichartz' estimate \eqref{eq:Str} to \eqref{eq:critpfm2}
to yield
\begin{equation}\label{eq:critpfm4}
\begin{aligned}
	\norm{\widetilde{v}_n^l}_{L^\rho((0,\I),L^\gamma)}
	\le{}& \norm{e^{it\Delta} \widetilde{v}_n^l(0)}_{L^\rho((0,\I),L^\gamma)}
	+ C_1 \norm{\widetilde{v}_n^l}_{L^\rho((0,\I),L^\gamma)}^p\\
	&{}+ C_1 \norm{\widetilde{f}_n^l}_{L^{\widetilde{\rho}'}((0,\I),L^{\widetilde{\gamma}'})},
\end{aligned}
\end{equation}
where the constant $C_1$ depends only on $N$ and $p$.
By Lemma \ref{lem:decay} and \eqref{eq:critpfm1}, we have
\begin{equation}\label{eq:critpfm5}
	\norm{e^{it\Delta} \widetilde{v}_n^l(0)}_{L^\rho((0,\I),L^\gamma)}
	\le C \ell (\widetilde{v}_n^l(0)) \le C_2 \eps,
\end{equation}
where $C_2$ depends only on $N$ and $p$.
We next estimate $\widetilde{f}_n^l$.
For almost all $t>0$, 
$\widetilde{v}^{l}_n (t)$ and $V_l(t)$ belong to $L^{\gamma}(\R^N)$
for all $n$ and $l$.
Fix such $t$.
For any fixed $\kappa>0$,
there exist compact sets $\Omega^j \subset \R^N$ ($j\ge1$)
such that
$\norm{V_j(t)}_{L^\gamma((\Omega^j)^c)} \le \kappa$.
Let
\[
	\Omega_n^l := \bigcup_{j=j_1+1}^l (\Omega^j - t\xi_n^j).
\]
We write $\widetilde{f}_n^l=F_1+F_2$ with
\[
	F_1 = |\widetilde{v}_n^l|^{p-1}\widetilde{v}_n^l
	-{\bf 1}_{\Omega_n^l}|\widetilde{v}_n^l|^{p-1}\widetilde{v}_n^l,
\]
\[
	F_2 = {\bf 1}_{\Omega_n^l}|\widetilde{v}_n^l|^{p-1}\widetilde{v}_n^l
	-\sum_{j=j_1+1}^l (|V_j|^{p-1} V_j )(t, x-t\xi^j_n) e^{i\xi_n^j\cdot x}
	e^{i\frac{t}2 |\xi_n^j|^2}
	.
\]
Since 
\begin{align*}
	\norm{\widetilde{v}^l_n}_{L^\gamma((\Omega_n^l)^c)}
	={}&  \norm{\sum_{j=j_1+1}^l V_j(t, x-t\xi^j_n) e^{i\xi_n^j\cdot x}
	e^{i\frac{t}2 |\xi_n^j|^2}}_{L^\gamma((\Omega_n^l)^c)} \\
	\le{}& \sum_{j=j_1+1}^l \norm{V_j(t,\cdot-t \xi_n^j)}_{L^\gamma((\Omega_n^l)^c)} 
	\le  (l-j_1)\kappa,
\end{align*}
we have
\[
	\norm{F_1}_{
	L^{\widetilde{\gamma}'}} 
	\le \norm{\widetilde{v}^l_n}_{L^\gamma((\Omega_n^l)^c)}^p
	\le (l-j_1)^p\kappa^p.
\]
Since $\Omega^j$ is compact and since $t>0$, one sees from \eqref{eq:critpf4} that
there exists $n_4$ such that if $n\ge n_4$ then
\[
	(\Omega^j - t\xi_n^j) \cap (\Omega^k - t\xi_n^k) = \emptyset
\]
for all $j_1< j<k\le l$.
Therefore, for $n\ge n_4$,
\begin{align*}
	{\bf 1}_{\Omega_n^l}|\widetilde{v}_n^l|^{p-1}\widetilde{v}_n^l
	&{}=\sum_{j=j_1+l}^l {\bf 1}_{(\Omega^j - t\xi_n^j)}|\widetilde{v}_n^l|^{p-1}\widetilde{v}_n^l\\
	&{}=\sum_{j=j_1+l}^l |{\bf 1}_{(\Omega^j - t\xi_n^j)}\widetilde{v}_n^l|^{p-1}{\bf 1}_{(\Omega^j - t\xi_n^j)}\widetilde{v}_n^l.
\end{align*}
Let us recall a well known estimate;
there exists constant $C>0$ depending only on $p>1$ such that
\[
	\abs{|a|^{p-1}a-|b|^{p-1}b} \le C(|a|^{p-1}+|b|^{p-1})\abs{a-b}
\]
for any $a,b \in \C$. By this estimate,
\begin{align*}
	|F_2| \le C\sum_{j=j_1+1}^l& \(|{\bf 1}_{(\Omega^j - t\xi_n^j)}\widetilde{v}_n^l|^{p-1}
	+|V_j(t, x-t\xi^j_n)|^{p-1} \)\\
	& \times
	\abs{{\bf 1}_{(\Omega^j - t\xi_n^j)}\widetilde{v}_n^l-V_j(t, x-t\xi^j_n) e^{i\xi_n^j\cdot x}
	e^{i\frac{t}2 |\xi_n^j|^2}}.
\end{align*}
It holds that
\begin{align*}
	&\Lebn{{\bf 1}_{(\Omega^j - t\xi_n^j)}\widetilde{v}_n^l-V_j(t, x-t\xi^j_n) e^{i\xi_n^j\cdot x}
	e^{i\frac{t}2 |\xi_n^j|^2}}\gamma\\
	&{}\le
	\norm{\widetilde{v}_n^l-V_j(t, x-t\xi^j_n) e^{i\xi_n^j\cdot x}
	e^{i\frac{t}2 |\xi_n^j|^2}}_{L^\gamma{(\Omega^j - t\xi_n^j)}}
	+ \norm{V_j(t)}_{L^\gamma({(\Omega^j)^c})} \\
	&{}\le \sum_{j_1<m\le l,\,m\neq j} \norm{V_m(t, x-t\xi^m_n)}_{L^\gamma{(\Omega^j - t\xi_n^j)}}
	+ \norm{V_j(t)}_{L^\gamma({(\Omega^j)^c})} 
	\le (l-j_1)\kappa,
\end{align*}
where we have used the fact that $\Omega^j - t\xi_n^j \subset (\Omega^m - t\xi_n^m)^c$ 
to prove the
last inequality.
It then follows that
\[
	\norm{F_2}_{
	L^{\widetilde{\gamma}'}} 
	\le C\sum_{j=j_1+1}^l \(\Lebn{\widetilde{v}_n^l(t)}{\g}^{p-1}
	+\Lebn{V_j(t)}{\gamma}^{p-1}\) (l-j_1)\kappa \le C\kappa.
\]
Hence, we finally obtain $\norm{\widetilde{f}_n^l}_{L^{\widetilde{\gamma}'}}
\le C(\kappa + \kappa^p)$ for any $n\ge n_4$.
Thus, for almost all $t>0$, $\Lebn{\widetilde{f}_n^l(t)}{\widetilde{\gamma}'} \to 0$ as $n\to\I$.
On the other hand,
\[
	\Lebn{\widetilde{f}_n^l(t)}{\widetilde{\gamma}'}
	\le \Lebn{\widetilde{v}_n^l(t)}\gamma^p +
	\sum_{j=j_1+1}^l \Lebn{V_j(t)}\gamma^p
	\le C(p,l) \sum_{j=j_1+1}^l \Lebn{V_j(t)}\gamma^p
\]
is valid for all $n$.
Since the right hand side is independent of $n$ and belongs to $L^{\widetilde{\rho}'}((0,\I))$,
we see from Lebesgue's convergence theorem that there exists $n_{1,2}$ such that
if $n\ge n_{1,2}$ then
\begin{equation}\label{eq:critpfm6}
	\norm{\widetilde{f}^l_n}_{L^{\widetilde{\rho}'}((0,\I),L^{\widetilde{\gamma}'})} \le \eps.
\end{equation}
Plugging \eqref{eq:critpfm5} and \eqref{eq:critpfm6} to
\eqref{eq:critpfm4}, we obtain
\[
	\norm{\widetilde{v}_n^l}_{L^\rho((0,\I),L^\gamma)}
	\le C_3 \eps + C_1 \norm{\widetilde{v}_n^l}_{L^\rho((0,\I),L^\gamma)}^p
\] 
for $n>n_1:=\max(n_{1,1},n_{1,2})$, where
 $C_3$ is independent of $l$ and $n$.
We choose $\eps = \eps(C_1,C_3)$ so small that this inequality gives
\[
	\norm{\widetilde{v}_n^l}_{L^\rho((0,\I),L^\gamma)}
	\le 2 C_3 \eps.
\]
For such $\eps$, we obtain the estimate
\[
	\norm{\widetilde{u}_n^l}_{L^\rho((0,\I),L^\gamma)}
	\le \sum_{j=1}^{j_1} \norm{V_j}_{L^\rho((0,\I),L^\gamma)}
	+ 2 C_3 \eps =: A
\]
for any $l>j_1$ and $n>n_1(l)$, which is \eqref{eq:critpf9}.

Let $\eps_0(A)$ be a number given by the long-time perturbation theory
(Proposition \ref{prop:lpt} (1)).
By \eqref{eq:critpf1} and \eqref{eq:critpf5}, 
there exists a number $l_1$ such that
if $l\ge l_1$ then we can choose $n_2(l)$ so that
\begin{equation}\label{eq:critpf10}
	\norm{e^{it\Delta} (u_{0,n} - \widetilde{u}_n^l(0))}_{L^\rho((0,\I),L^\gamma)}
	= \norm{e^{it\Delta} W_n^{l}}_{L^\rho((0,\I),L^\gamma)}
	\le \eps_0(A)
\end{equation}
for all $n\ge n_2(l)$.
We fix $l> \max(j_1,l_1)$.
Arguing as in the proof of \eqref{eq:critpfm6},
we see that
there exists $n_3(l)$ such that
\begin{equation}\label{eq:critpf11}
	\norm{\widetilde{e}^l_n}_{L^{\widetilde{\rho}'}((0,\I),L^{\widetilde{\gamma}'})} \le 
	\eps_0(A)
\end{equation}
holds for $n\ge n_3(l)$.

We are now in a position to complete the proof of the claim.
Choose $l> \max(j_1,l_1)$ and $n\ge \max(n_1(l),n_2(l),n_3(l))$.
Using the long-time perturbation theory,
we deduce from \eqref{eq:critpf9}, \eqref{eq:critpf10}, and \eqref{eq:critpf11}
that
$\norm{u_n}_{L^\rho((0,\I),L^\gamma)}\le c(A)<\I$, where $u_n$ is a solution to \eqref{eq:NLS}
with $u_n|_{t=0}=u_{0,n}$.
Thanks to Proposition \ref{prop:ScatterCond}, this implies
$u_{0,n} \in S_+$.
However, this contradicts with the definition of $u_{0,n}$. 
\end{proof}
\begin{proof}[Proof of Theorem \ref{thm:main1}]
Choose a sequence $\{u_{0,n}\}_n \subset \F H^1\setminus S_+$
so that $\ell(u_{0,n})\le \ell_c + \frac1n$.
By scaling, we can assume $\Lebn{u_{0,n}}2=1$.
We now apply Lemma \ref{lem:oneprofile}.
Then, there exist a subsequence of $\{u_{0,n}\}$, which is denoted again by $\{u_{0,n}\}$,
a function $\psi \in \F H^1$ with $\ell(\psi)=\ell_{c}$ and
$\Lebn{\psi}2=1$, and 
sequences $\{W_n\}_{n}\subset \F H^1$ and $\{\xi_n\}_{n} \subset \R^N$ such that
\eqref{eq:cc1} and \eqref{eq:cc2} hold.

If $\psi\in S_+$ then it follows from Proposition \ref{prop:ScatterCond} that
$\norm{V}_{L^\rho((0,\I),L^\gamma)}<\I$, where $V$ is a solution of 
\eqref{eq:NLS} with $V(0)=\psi$.
Now, apply Proposition \ref{prop:lpt} with $\widetilde{u}(t,x)=V(t,x-\xi_n t)
e^{i\xi_n\cdot x}e^{i\frac{t}2|\xi_n|^2}$.
Remark that $e\equiv0$. One also verifies from \eqref{eq:cc2} that
\[
	\norm{e^{it\Delta} (u_{0,n}-\widetilde{u}(0))}_{L^\rho((0,\I),L^\gamma)}
	\le \norm{e^{it\Delta} W_n}_{L^\rho((0,\I),L^\gamma)} \to 0
\]
as $n\to0$. Hence, by means of the long-time perturbation theory,
we see $u_{0,n} \in S_+$ for large $n$,
which is a contradiction.
Thus, $\psi\in \F H^1 \setminus S_+$.

Finally, we prove $	\ell_{c}= \inf\{ \ell (f) \ |\ f \in \F H^1 \setminus S \}$.
For this, it suffices to show
\begin{equation}\label{eq:end1}
	\ell_{c}
	= \inf_{f \in \F H^1 \setminus S_-} \ell (f)
\end{equation}
since
\[
	\inf_{f \in \F H^1 \setminus S} \ell (f)
	= \min \(\inf_{f \in \F H^1 \setminus S_+} \ell (f)
	,\inf_{f \in \F H^1 \setminus S_-} \ell (f)\).
\]
Let us now recall that
if $u(t,x)$ is a solution then $\overline{u}(-t,x)$ is also a solution.
This implies $\overline{u_0}\in S_-$ if and only if $u_0 \in S_+$.
If $\ell_c > \inf\{ \ell (f) \ |\ f \in \F H^1 \setminus S_- \}$
then there exists $w_0 \in \F H^1 \setminus S_-$ such that $\ell(w_0)<\ell_c$.
However, it then holds that
$\overline{w_0}\in \F H^1 \setminus S_+$ and $\ell(\overline{w_0})<\ell_c$,
which contradicts to the definition of $\ell_c$.
Hence, $\ell_c \le \inf\{ \ell (f) \ |\ f \in \F H^1 \setminus S_- \}$.
A similar argument shows $\ell_c \ge \inf\{ \ell (f) \ |\ f \in \F H^1 \setminus S_- \}$.
We obtain \eqref{eq:end1}.
\end{proof}

\subsection*{Acknowledgments}
The author expresses his deep gratitude
to Professor Masahito Ohta for fruitful discussions. 
The author also thanks Professor Masaya Maeda for giving him valuable comments
on preliminary version of the manuscript.
This research is supported by Japan Society for the Promotion of Science(JSPS)
Grant-in-Aid for Young Scientists (B) 24740108.

\providecommand{\bysame}{\leavevmode\hbox to3em{\hrulefill}\thinspace}
\providecommand{\MR}{\relax\ifhmode\unskip\space\fi MR }
\providecommand{\MRhref}[2]{%
  \href{http://www.ams.org/mathscinet-getitem?mr=#1}{#2}
}
\providecommand{\href}[2]{#2}

\end{document}